\documentclass[a4paper,12pt]{article}
\usepackage{bm}
\usepackage{amsmath}
\usepackage{amsthm}
\usepackage{wasysym}
\usepackage{amssymb}
\usepackage{amsfonts}
\usepackage{graphicx}
\usepackage[english]{babel}
\usepackage[utf8]{inputenc}
\usepackage[T1]{fontenc}
\usepackage{mathrsfs}
\usepackage{enumerate}
\usepackage{version}
\usepackage{calc}
\usepackage{subfigure}
\usepackage{authblk}
\usepackage{comment}

 \usepackage{pstricks,pst-math,pst-xkey}
 \usepackage{float}
 \usepackage{indentfirst}
\usepackage[pagewise]{lineno} 
\newtheorem{definition}{Definition}[section]
\newtheorem{theorem}[definition]{Theorem}
\newtheorem{lemma}[definition]{Lemma}
\newtheorem{corollary}[definition]{Corollary}
\newtheorem{proposition}[definition]{Proposition}

\numberwithin{equation}{section}

\newcommand{\N}{{\mathbb N}}

\newcommand{\R}{{\mathbb R}}

\newcommand{\ve}{{\varepsilon}}

\def\ve{\varepsilon}

\def\f{\varphi}

\def\d{\partial}

\def\ti{\tilde}
\def\tw{\tilde{w}}
\def\w{\omega}

\def\toW{\tilde{\overline{W}}}
\def\N{\mathbb{N}}

\def\F{\mathcal{F}}
\def\T{\mathbb{T}_L}

\def\P{\mathbb{P}}

\def\f{\varphi}
\def\Z{\mathbb{Z}}

\def\R{\mathbb{R}}
\def\E{\mathbb{E}}
\def\N{\mathbb{N}}
\def\P{\mathbb{P}}

\begin{document}

\title{Thin Film Equations with Nonlinear Deterministic and Stochastic Perturbations}

\author[1]{Oleksiy Kapustyan \thanks{kapustyanav@gmail.com}}

\author[3] {Olha Martynyuk
\thanks{o.martynyuk@chnu.edu.ua}}

\author[2]{Oleksandr Misiats\thanks{omisiats@vcu.edu}}

\author[1]{Oleksandr Stanzhytskyi \thanks{ostanzh@gmail.com}}

\affil[1]{Department of Mathematics,
Taras Shevchenko National University of Kyiv, Ukraine}

\affil[2]{Department of Mathematics, Virginia Commonwealth University,
Richmond, VA, 23284, USA}

\affil[3]{Department of Mathematics, Yuriy Fedkovych Chernivtsi National University, Chernivtsi, Ukraine}


\maketitle

\begin{abstract}
In this paper we consider stochastic thin-film equation with nonlinear drift terms, colored Gaussian Stratonovych noise, as well as nonlinear colored Wiener noise. By means of Trotter-Kato-type decomposition into deterministic and stochastic parts, we couple both of these dynamics via a discrete-in-time scheme, and establish its convergence to a non-negative weak martingale solution. 
\end{abstract}

\section{Introduction.} We consider the nonlinear stochastic thin-film  equation with nonlinear drift coefficients
\begin{equation}\label{1.1}
du = (-\partial_x(u^{2}u_{xxx})  + l(u))dt + \partial_x (u \circ dW)+ f(u)dW_1(t)
\end{equation}
for $(t,x) \in [0,T) \times \mathbb{T}_L$, on torus $\mathbb{T}_L$, where $T$ and $L$ are positive constants, and $\mathbb{T}_L$ denotes the torus on the interval $[0,L]$ with periodic boundary conditions
\[
\partial_x^i u(\cdot, 0) = \partial_x^i u(\cdot, L), \ i = 0,1,2,3
\]
and non-negative initial condition $u(0,x) = u_0(x)$. The term $\partial_x (u \circ dW)$ is a stochastic perturbation in Stratonovich form, and $f(u) dW$ is a Stochastic perturbation of Ito type. Here 
\[
W(t,x) := \sum_{k \in \mathbb{Z}} \lambda_k \Psi_k(x) \beta^k(t), \ \ W_1(t,x) := \sum_{k \in \mathbb{Z}} \gamma_k \Psi_k(x) \beta_1^k(t),
\]
where $\{\Psi_k\}$ is ONB in $H^2(\mathbb{T})$, namely
\begin{equation}\label{eq:ONB}
\Psi_k(x) := c_k 
\begin{cases}
\cos \left(\frac{2\pi k}{L} x\right), k > 0, \ x \in [0,L];\\
\frac{1}{\sqrt{2}}, k = 0, \ x \in [0,L];\\
\sin \left(\frac{2\pi k}{L} x\right), k < 0, \ x \in [0,L],
\end{cases}
\end{equation}
with
\[
c_k = \sqrt{\frac{2}{L\left(1 + \left(\frac{2 \pi k}{L}\right)^2 + \left(\frac{2 \pi k}{L}\right)^4 \right)}}.
\]
In particular, the functions $\{\Psi_k\}$ are the eigenfunctions of the Laplace operator with periodic boundary conditions, which satisfy
\begin{eqnarray}\label{1.5}
& & \partial_x \Psi_k = \frac{2\pi k}{L} \Psi_{-k}; \ \partial_x^2 \Psi_k = -\frac{4\pi^2 k^2}{L^2} \Psi_{k};\\
\nonumber & &\partial_x^3 \Psi_k = -\frac{8\pi^3 k^3}{L^3} \Psi_{-k}; \ \partial_x^4 \Psi_k = \frac{16\pi^4 k^4}{L^4} \Psi_{k}, \ k \in \Z.
\end{eqnarray}

The processes
$\beta^k$ and $ \beta_1^k$ are mutually independent standard real-valued $\mathcal{F}_t$-Wiener processes on a complete filtered probability space $(\Omega, \mathcal{F}, \{\mathcal{F}_t\}, \mathcal{P}), t \in [0,T]$, with a complete and right-continuous filtration $(\mathcal{F}_t)_{t \in [0,T]}$. The coefficients $\lambda_k\geq 0$ and $\gamma_k$ satisfy the coloring condition 
\begin{equation}\label{1.3}
    \sum_{k \in \mathbb{Z}} (\lambda_k^2 + \gamma_k^2) < \infty.
\end{equation}
Since in this work we will be studying the martingale solutions using the Skorokhod approach \cite{Skor} and its generalization \cite{Jakub}, without loss of generality we consider the probability space to be $\Omega = [0,1]$ with Lebesgue measure. Similarly to \cite{Gess}, may be re-written in Ito's form 
\begin{multline*}
du = \left(\partial_x(-u^{2}u_{xxx}) + \frac{1}{2} \sum_{k=1}^{\infty} \lambda_k^2 \Psi_k^2 \partial_x (\Psi_k u ) + l(u)\right)dt + \sum_{k=1}^{\infty} \lambda_k (\partial_x(\Psi_k u)) d \beta^k \\
+ \sum_{k=1}^{\infty} \gamma_k \Psi_k f(u) d \beta_1^k
\end{multline*}
Finally, we will assume that the nonlinear drift coefficients are $l(u) = -|u|^{r-1} u$ for some $r \geq 1$, and $f(u)$ is globally Lipschitz with $f(0) = 0$. 

The deterministic equations of type \eqref{1.1}  arise in  modeling the motion of liquid droplets of thickness $u$, spreading over the solid surface. This model follows from lubrication theory under the assumption that the dimensions in the horizontal directions are significantly larger than in the vertical (normal) one. In this regime the dynamics of the droplet is governed by the surface tension and limited by viscosity.  

In broad terms, the dynamics of the deterministic version of \eqref{1.1} is characterized by the presence of the {\it wetted regions} $u>0$. The equation is parabolic in the interior of these regions, and degenerate on their boundary. The boundary of the wetted region, in turn, has a finite speed of propagation \cite{Bernis}. Thus, one may interpret \eqref{1.1} as a fourth-order nonlinear free boundary problem inside a wetted region, which itself evolves in time. Furthermore, the classical parabolic theory is not applicable to this equation, in particular, due to the lack of comparison (maximum) principle, which is widely used in the existence theory of degenerate second order parabolic equations. Bernis and Friedman \cite{Bern} used the energy-entropy method to prove the existence of a non-negative generalized weak
solution, which was constructed
as a limit of solutions of a regularized problem. In this work, the authors' notion of a weak solution is somewhat ``weaker'' than usual, since the integral identity in the definition of the weak solution has to hold not in the entire $\mathbb{T}$, but only on the subset where $u>0$. In addition, in  \cite{Bern} the authors analyzed the support of the solution. The existence of more regular (strong or entropy) solutions was shown in \cite{DalPas}, where the authors also studied the asymptotic in $t$ behavior of the solution. Further properties of solutions, including convergence to steady states, finite propagation speed, and waiting time phenomenon, were obtained in \cite{Bert}.  
The work \cite{KapTar} studies the generalized thin film equation with a nonlinear dissipative term $l(u)$, which models the relation between nonlinear absorption and spatial injection.

In this paper we consider the thin film equation with two different stochastic perturbations, of Ito and Statonovych type. It is worth mentioning that one needs to be very careful when dealing with the effect of noise on nonlocal and/or ill-posed problems. In some cases, e.g. \cite{MisStaTop}, the presence of a even small stochastic perturbation leads to a finite time blowup while an unperturbed equation has global solution. In others \cite{Sta}, the effect of the random perturbation is exactly opposite - it may lead to the existence of a global solution while the corresponding deterministic equation has a finite time blowup. The long time behavior of stochastically perturbed evolution equations is typically described via the existence and properties of invariant measures, see, e.g. \cite{MisStaYip1}, \cite{MisStaYip2}, \cite{MisStaYip3}, \cite{MisStaSta}, \cite{MisStaHie}, \cite{MisStaKap}, \cite{MisStaCla}.

The stochastic version of thin-film equation, which takes into account the effect of random forcing when modeling the enhanced spreading of droplets, was first introduced in \cite{17}, with $l = f = 0$. In the subsequent work \cite{33} the authors additionally take the interface potential between fluid and substrate into account, which prevents the solution $u$ from becoming negative. This work describes coarsening and de-wetting phenomena. The first rigorous construction of a non-negative martingale solution of the stochastic thin film equation with Ito noise and additional interface potential was obtained in \cite{19}. In \cite{13}, the author considered a more general case of the main operator in the form $-\partial_x(u^{n}u_{xxx})$ (referred as more general {\it{mobility}}), and established the conditions for the existence of a global strong solution for this problem. 
In \cite{Gess} the authors established the existence of a nonnegative matringale solution for \eqref{1.1} by means of Trotter-Kato type decomposition. The subsequent work \cite{Gess2} establishes the existence and uniqueness of a weak solution of this problem. To this end, the authors started with the establishing the existence of a weak solution for the regularized equation by means of Galiorkin approximations, and proceeded with passing to the limit in the regularized problem.   \\

The equation \eqref{1.1}, which is considered in this paper, has significant differences from the similar models, analyzed by the other authors. In it, we introduce the nonlinear drift $l(u)$, as well as the nonlinear stochastic Ito perturbation $f(u) d W_1(t)$. Due to the presence of these two terms, the equation is no longer in divergence form. In other words, if $f \equiv 0$ and $l \equiv 0$, integrating both sides of \eqref{1.1} over $\mathbb{T}$,  one gets
$\frac{d}{dt} \int_{\mathbb{T}} u dx =0$, implying, in view of nonnegativity of $u$, the conservation of mass property (almost surely). This property plays a crucial role in \cite{Gess}. However, if either $f\neq 0$ or $l \neq 0$, this property no longer holds. Nevertheless, in this paper we obtain the estimates on the mass, which, in turn, enable us to obtain the energy estimates similar to \cite{Gess}. Our analysis starts with Trotter-Kato decomposition, which separates the deterministic dynamics from stochastic. The principal difference in our case is the presence of the nonlinear drift term in the deterministic dynamics, which prevents us from using the classic results on deterministic thin-film equations, e.g. \cite{Bern, DalPas}. Instead, we obtain the analogs of the main result in \cite{KapTar} with different boundary conditions. The stochastic dynamics, in contrast to \cite{Gess}, is nonlinear, which makes its analysis more complicated, especially in view of the lack of mass preservation. In particular, it requires a different approach to establish the non-negativity property of the solutions. 

The paper is structured as follows: in Section \ref{Sec2} we introduce the notation, list preliminary results, formulate the main results, and introduce the decomposition of the dynamics. In Section \ref{Sec3} we describe the deterministic dynamics in detail. Section \ref{Sec4} is devoted to stochastic dynamics, and finally, the main result is established in Section \ref{Sec5}.

\section{Preliminaries and main results.}\label{Sec2}
Throughout the paper, we denote
\[
Q_T = (0,T) \times \mathbb{T}_L.
\]
For $u,v \in \mathbb{T}_L$, let
\[
(u,v)_2 := \int_0^L u(x) v(x) dx, \text{ and } \|u\|_2:= \sqrt{(u,u)_2}.
\]
Next, for $Q \in \R^d$ with $\partial Q \in C^{\infty}$, for $s \in [0, \infty)$, $p \in [1, \infty)$, let $W^{s,p}(Q)$ be the regular Sobolev space for $s \in \N$, and Sobolev-Slobodeckij space for non-integer $s$. For $p=2$ we will denote $W^{s,p}(Q) = H^s(Q) = H^s$. If $X$ is a Banach space, the space $C^{k+\alpha}(Q;X)$ is the space of $k$ times differentiable functions $Q \to X$, whose $k$-th derivatives are Holder-continous with exponent $\alpha \in (0,1)$ on compact subsets of $Q$. We also denote $C^{k-}(Q;X)$ to be the space of $k-1$-times differentiable functions, whose $k-1$-st derivatives are Lipschitz continuous. The space $BC^{0}(Q,X)$ is the set of bounded continuous functions. The pairings $\langle \cdot, \cdot \rangle$ is the $H^{-1}(\mathbb{T}_L) - H^{1}(\mathbb{T}_L)$ pairing in $L^2(\mathbb{T}_L)$, and $\langle \langle \cdot, \cdot \rangle \rangle$ stands for the pairing between $L^2(\mathbb{T}_L)$ and $H^2(\mathbb{T}_L)$ in $H^1(\mathbb{T}_L)$. We denote $H^1_w(\mathbb{T}_L)$ to be the space of  $H^1(\mathbb{T}_L)$ functions endowed with the weak topology, induced by $\| \cdot \|_{1,2}$. The space $\dot{H}^1(\T)$ is the Sobolev space of $H^1(\T)$ functions endowed with the norm $\|\d_x u\|_{L^2(\T)}$. In this space the functions which differ by an additive constant are the same. 
Denote $L_2(U;H)$ to be the space of Hilbert-Schmidt operators from $U$ to $H$. In this case, for $A \in L_2(U;H)$ we have
\[
\|A\|^2_{L_2(U;H)}: = \sum_{n \geq 1} \|A e_n\|^2_H,
\]
where $\{e_n\}$ is any orthonormal basis in $U$. The symbol $\langle \langle \cdot \rangle \rangle_t$ denotes the quadratic variation process.
For  $u: Q_T \to \R$ we denote 
\[
P_T := \{(t,x) \in \bar{Q}_T, u(t,x)>0\}.
\]
Consider a triple consisting of a filtered probability space $([0,1], \tilde{\mathcal{F}}, \tilde{\mathcal{F}}_{t \in [0,T]}, P)$, where $\tilde{\mathcal{F}}_{t \in [0,T]}$ is a complete right continuous filtration.
\begin{definition}\label{Def2.1}
Let $u_0 \in H^1(\mathbb{T}_L)$. An  $\tilde{\mathcal{F}}_{t}$ adapted bounded continuous $H^1_w(\mathbb{T}_L)$ - valued process $\tilde{u}$ on $[0,T]$, such that the distributional derivative $\partial_x^3 \tilde{u}$ is $\ti{\F}_t$ - adapted with $\partial_x^3 \tilde{u} \in L^2(\{\ti{u}>r\})$ for any $r>0$, and $\ti{u}^2(\partial_x^3 \tilde{u})$ is in $L^2(\{\ti{u}>0\})$ $\P$ - almost surely, as well as mutually independent standard real-valued $(\ti{\F}_t)$-Wiener processes $\tilde{\beta}^k$, is called a {\bf martingale solution} of the equation \eqref{1.1}, if its weak formulation
\begin{align}\label{2.1}
    & & (\ti{u}(t,\cdot), \f)_2 -  (u_0, \f)_2 = \int_0^t  \int_{\{\ti{u}(s,\cdot)>0\}} \ti{u}^2(s,\cdot) (\d_x^3 \ti{u}(s,\cdot)) \d_x \f dx ds  \\
    \nonumber &  & - \frac{1}{2} \sum_{k \in \Z} \lambda_k^2 \int_0^t ( \Psi_k \d_x(\Psi_k \ti{u}(s, \cdot)), \d_x \f)_2 ds + \int_0^t (l(\ti{u}(s, \cdot)), \f)_2 ds \\
    \nonumber & & - \sum_{k \in \Z} \lambda_k \int_0^t (\Psi_k \ti{u}(s,\cdot), \d_x \f)_2 d \ti{\beta}^k(s)  + \sum_{k \in \Z} \lambda_k \int_0^t (\Psi_k f(\ti{u}(s,\cdot)), \f) d \ti{\beta}^k(s)
\end{align}
is satisfied for every $\f \in C^{\infty}(\T)$ and $t \in [0,T]$ $\P = \lambda_{[0,1]}$ almost surely.
\end{definition}
We have the following result:
\begin{theorem}\label{Thm2.1}
(existence of martingale solution) Suppose $u_0 \in H^1(\T)$ is such that $u_0 \geq 0$. Then the
equation \eqref{1.1} with has a martingale solution 
$\ti{u}(t)$ in the sense of Definition \ref{Def2.1}, which is non-negative a.s. for $t \in [0,T]$, and for any $p \geq 2$ there is $C_p>0$ such that
\[
\mathbb{E} \sup_{t \in [0,T]} \|\ti{u}(t, \cdot)\|_{H^1(\mathbb{T})}^p \leq C_p  \|u_0\|_{H^1(\mathbb{T})}^p
\]
for any $p \in [2, \infty)$, where $C<\infty$ is independent on $u_0$.
\end{theorem}

In order to establish this result, we will use the Trotter-Kato type decomposition of the dynamics into deterministic and stochastic parts. This method was used by a number of authors, e.g. \cite{Gess, ManZau}, in particular, to establish the existence of for SPDEs with local Lipschitz coefficients. To this end, for fixed $N\geq 1$, we equi-partition the time interval $[0,T]$ into intervals of length $\delta = \frac{T}{N+1}$. Then, for $t \in [(j-1) \delta, j\delta]$ and for an arbitrary $\f \in C^{\infty} (\mathbb{\T})$ we define the following dynamics:\\

{\bf Deterministic Dynamics (D)} We look for the function $v_N$ satisfying
\begin{multline}\label{DynamicsD}
(v_N(t, \cdot), \f)_2 - (v_N((j-1)\delta, \cdot), \f)_2 = \int_{(j-1)\delta}^t \int_{v_N(s, \cdot) > 0} v_N^2(s, \cdot) \partial_x^3 v_N^2(s, \cdot) (\partial_x \f) dx ds  \\
+ \int_{(j-1)\delta}^t (l(v_N(s,\cdot)), \f)_2 ds.
\end{multline}
\\

{\bf Stochastic Dynamics (S)} We look for the function $w_N$ satisfying
\begin{multline}\label{DynamicsS}
(w_N(t, \cdot), \f)_2 - (w_N((j-1)\delta, \cdot), \f)_2 = - \frac{1}{2}  \sum_{k \in \mathbb{Z}} \lambda_k^2 \int_{(j-1)\delta}^t (\Psi_k \partial_x(\Psi_k w_N(s, \cdot)), \partial_x \f)_2 ds  \\
 - \sum_{k \in \mathbb{Z}} \lambda_k \int_{(j-1)\delta}^t (\Psi_k w_N(s, \cdot), \partial_x \f)_2 d \beta^k(s) + \sum_{k \in \mathbb{Z}} \gamma_k \int_{(j-1)\delta}^t (\Psi_k f(w_N(s, \cdot)), \f)_2 d \beta_1^k(s).
\end{multline}

{\bf Deterministic-Stochastic Connection (DS)} We set $v_N(0):=u_0$, \\
$v_N(j \delta, \cdot) := \lim_{t \to j\delta-} w_N(t, \cdot)$ and $w_N((j-1) \delta, \cdot) := \lim_{t \to j\delta-} v_N(t, \cdot)$ a.s.\\

The dynamics in (D) is deterministic with random initial conditions, while the dynamics in (S) is purely stochastic. We note that both deterministic and stochastic dynamics are significantly different from the one considered in \cite{Gess} due to the presence of nonlinearities $l$ and $f$ in them respectively. In particular, $f$ makes the stochastic dynamics nonlinear. In order to resolve the lack of mass preservation property, our strategy is to show that, due to the dissipative nature of nonlinearity $l(u)$, the mass in deterministic dynamics is non-increasing in time, while the stochastic dynamics preserves the mass on average, i.e. $\mathbb{E} \int_{\mathbb{T}} w_N(x) dx \equiv const.$ We proceed with establishing the apriori bounds on the solutions of both deterministic and stochastic problems, which enable us to pass to the limit as $N \to \infty$, thus establishing that both $v_N$ and $w_N$ converge to the desired martingale solution for \eqref{1.1}.

\section{Deterministic dynamics.}\label{Sec3} In this section we consider the deterministic equation 
\begin{equation}\label{eqDet}
    \begin{cases}
        u_t + \d_x (u^2 \d_x^3 u) + |u|^{\lambda-1} u =0 \text{ in } (0,T) \times \T \\
        u(0,x) = u_0(x),
    \end{cases}
    \end{equation}
    with $\lambda \geq 1$, and nontrivial non-negative $u_0 \in H^1(\T)$.
\begin{definition}\label{Def3.1}
   The function $u \in C(\bar{Q}_T) \cap L^{\infty}(0,T, H^1(\T))$ is called the weak solution of \eqref{eqDet} if
   \begin{enumerate}
       \item $u \in C^{1,4}(P_T)$, $u \cdot u_{xxx} \in L^2(P_T)$, where 
       \[
       P_T = \{(t,x) \in \bar{Q}_T: u>0\};
       \]
       \item 
       \[
       \int_{Q_T} u \cdot \psi_t dx dt + \int_{P_T} u^2 u_{xxx} \psi_x dx dt - \int_{Q_T} |u|^{\lambda - 1} u \cdot \psi dx dt = 0
       \]
       for all $\psi \in C^1(\bar{Q}_T)$, and $u(t,\cdot) \to u_0(\cdot)$ in $H^1(\T)$ as $t \to \infty$.
   \end{enumerate}
\end{definition}
The equation of type \eqref{eqDet} was studied in \cite{KapTar} with the nonlinear term in the form $(|u|^n u_{xxx})_x, \ n \geq 1,$ and with the boundary conditions $u_x|_{x = 0, L} = u_{xxx}|_{x = 0,L} = 0$ on $[0,T]$. However, the main result in \cite{KapTar} guarantees the existence of the solution for arbitrary $0 \leq u_0(x) \in H^1(\T)$ only for $n \in (1,2)$. For other values of $n$, the solution exists only under the initial conditions, which satisfy the entropy estimates. Therefore, the existence of the solution of \eqref{eqDet} for any $0 \leq u_0(x) \in H^1(\T)$ requires extra work. We have the following result:
\begin{theorem}\label{Thm3.1}
    The Cauchy problem \eqref{eqDet} admits a nonnegative solution in $Q_T$. Furthermore,
    \[
      \int_0^L u(t,x) dx \leq \int_0^L u_0(x) dx \text{ on } t \in [0,T].
    \]
\end{theorem}
\begin{proof}
For any $\ve>0$ introduce $f_\ve(u) = \frac{u^6}{\ve u^2+u^4}$. Suppose $u_{0 \ve} > 0$ in $[0,L]$ and $u_{0 \ve}(x) \in C^\infty(\T)$ and $u_{0 \ve} \to u_0$ in $H^1(\T), \ \ve \to 0$. Thus, it follows from the results in \cite{27} that the regularized problem
\begin{equation}\label{eqDetReg}
    \begin{cases}
        u_t + \d_x (f_\ve(u) \d_x^3 u) + |u|^{\lambda-1} u =0 \text{ in } (0,T) \times \T \\
        u(0,x) = u_0 \ve(x),
    \end{cases}
    \end{equation}
    has a unique non-negative smooth (classic) solution $u^\ve(t,x)$ in $Q_T$. This fact is proved using compactness arguments and parabolic Schauder estimates. Multiplying \eqref{eqDetReg} by $-u_{xx}^\ve$ and integrating over $(0,L)$ we have
    \begin{equation}\label{Identity1}
\frac{1}{2} \frac{d}{dt} \|u_x^2(t)\|_2^2 + \int_0^L f_\ve(u^\ve)(u_{xxx}^\ve)^2 dx + \lambda \int_0^L (u^\ve)^{\lambda - 1} (u_x^\ve)^2 dx = 0
    \end{equation}
The convergence of $u_{0\ve}$ implies the existence of a constant $C>0$ such that $\|u_x^\ve(0)\| \leq C$. Then it follows from \eqref{Identity1}, Poincare inequality, and the embedding $H^1(\T) \in L^\infty(\T)$: 
\begin{equation}\label{3.6}
    \|u^\ve\|_{L^\infty(Q_T)} \leq C_1, \ \ C_1 >0.
\end{equation}
Then  \eqref{3.6} yields $\forall x_1 < x_2 \in [0,L]$
\begin{equation}\label{3.7}
|u(t,x_1) - u(t,x_2)| \leq \int_{x_1}^{x_2}|u_x(t,x)|dx \leq K|x_1-x_2|^{\frac{1}{2}}. 
\end{equation}
Denote $h_\ve(u^\ve) := (f_\ve(u^\ve) u_{xxx}^\ve)$. Integrating \eqref{Identity1} on $(0,T)$ we get
\begin{equation}\label{Identity2}
    \frac{1}{2} \|u_x^\ve(T)\|_2^2 + \int_{Q_T} h^2_\ve(u^\ve) dx dt + \lambda \int_{Q_T} (u^\ve)^{\lambda-1} (u_x^\ve)^2 dx dt = \frac{1}{2} \|u_x^\ve(0)\|_2^2.
\end{equation}
Since $u^\ve$ is bounded in $L^\infty(Q_T)$, so is $f_\ve(u^\ve)$, thus
\begin{equation}\label{bound2}
    \int_{Q_T} h_\ve^2(u^\ve) dx dt \leq A \text{ for some } A>0.
\end{equation}
\begin{lemma}
There exists $M\geq 0$ such that for any $x_0 \in (0,L)$ and any $t_1, t_2 \in [0,T], \ |t_1 - t_2| \leq 1$
\begin{equation}\label{3.10}
    |u^\ve(t_1,x_0) - u^\ve(t_2, x_0)| \leq M |t_1-t_2|^{\frac{1}{8}}.
\end{equation}
\end{lemma}
\begin{proof}
In this proof we will follow the main ideas of \cite{Bern}. We argue by contradiction. In the identity
\begin{equation}\label{eq:weak}
    \int_{Q_T} u^\ve \f_t dx dt = -\int_{Q_T} f_\ve(u^\ve) u^\ve_{xxx} \f_x dx dt + \int_{Q_T} l(u^\ve) \f dx dt. 
\end{equation}
we set $\f(t,x) = \xi(x) \theta_\delta(t)$, where $|\xi(x)|\leq 1$ and
\begin{equation*}
    \xi(x) = 
    \begin{cases}
        0, \ |x-x_0| \geq C M^2 |t_2 - t_1|^{\frac{1}{4}},\\
        1, \ |x-x_0| \leq \frac{1}{2} C M^2 |t_2 - t_1|^{\frac{1}{4}},
    \end{cases}
\end{equation*}
and the function $|\theta_\delta| \leq 1$ is described in \cite{Bern}. Following \cite{Bern} we have
\[
\int_{Q_T} u^2 \f_t dx dt \geq C_3 M^3 (t_2-t_1)^{\frac{3}{8}},
\]
while
\begin{multline*}
   \left|\int_{Q_T} f_\ve(u^\ve) u^\ve_{xxx} \f_x dx dt + \int_{Q_T} l(u^\ve) \f dx dt \right| \\ 
   \leq \frac{C_4}{M} \left(\int_{Q_T} h_\ve(u^\ve) dx dt\right)^{\frac{1}{2}}(t_2 - t_1)^{\frac{1}{2} - \frac{1}{8}} + C_5 \|l\|_{\infty} M (t_2 - t_1)^{\frac{1}{2} + \frac{1}{8}}.
\end{multline*}
Hence
\[
M^3(t_2 - t_1)^{\frac{3}{8}}  \leq C_6\left(\frac{1}{M} (t_2-t_1)^{\frac{1}{2}-\frac{1}{8}} + M (t_2-t_1)^{\frac{1}{2}+\frac{1}{8}} \right) \leq C_6 \left(\frac{1}{M} + M\right) (t_2 - t_1)^{\frac{3}{8}}.
\]
Thus, any $M \geq 1$ has to satisfy the inequality $M^2 \leq 2 C_\delta$, which is a contradiction. Hence \eqref{3.10} follows. 
\end{proof}
It follows from \eqref{3.6}, \eqref{3.7} and \eqref{3.10} that the set $\{u^\ve\}$ is relatively compact in $C(\bar{Q}_T)$, and hence up to a subsequence $u^\ve \to u$ in $C(\overline{Q_T})$. Note that all the derivatives $u_t^\ve, u_t^\ve, u_x^\ve, u_{xx}^\ve, u_{xxx}^\ve$ and $u_{xxxx}^\ve$ converge to the corresponding derivatives of $u$ in view of uniform parabolicity of in any compact subset of $P_T$. Thus, passing to the limit in \eqref{eq:weak} as $\ve \to 0$, we deduce that $u$ solves \eqref{eqDet}, which is also non-negative. Let us show that the mass of the solution is non-increasing. Integrating the solution of the regularized problem \eqref{eqDetReg} on $[0,L]$, we have
\[
\frac{d}{dt} \int_0^L u^\ve(t,x) dx + \int_0^L (u^\ve(t,x))^\lambda dx  = 0
\]
Since $u \geq 0$, we have 
\[
\frac{d}{dt} \int_0^L u^\ve(t,x) dx \leq 0,
\]
hence, for the regularized problem,
\begin{equation}\label{nonincrease}
\int_0^L u^\ve(t,x) dx \leq \int_0^L u_{0\ve}(x) dx. 
\end{equation}
Since $u^\ve(t,x)$ converges to $u(t,x)$ uniformly in $C(\bar{Q}_T)$, while $u_{0 \ve}$ converges to $u_0$ in $H^1(\T)$, the non-increasing property of the mass of $u$ follows from \eqref{nonincrease}. Hence the proof is complete. 
\end{proof}
\begin{corollary}
Suppose the conditions of Theorem \ref{Thm3.1} hold. Then for any $p \geq 2$ the solution $u$ satisfies
\begin{equation}\label{pth_est}
    \| \d_x u(t, \cdot)\|_2^p + 2 \int_0^t \|\d_x u(s, \cdot)\|_2^{p-2} \int_{u(s, \cdot) > 0} u^2(s,x) (\d_x^3 u(s,x))^2 dx ds \leq \|\d_x u_0\|_2^p.
\end{equation}
\end{corollary}
\begin{proof}
    Since $u^\ve>0$, it follows from \eqref{eqDetReg} that
    \[
\|u_x^\ve(t, \cdot)\|_2^2  + 2 \int_{Q_t} f_\ve(u^\ve) (u^\ve_{xxx})^2 dx ds \leq \|u_x^\ve(0, \cdot)\|_2^2.
    \]
Thus for any $r>0$ we have 
\begin{equation}\label{quad_est}
\|u_x^\ve(t, \cdot)\|_2^2  + 2 \int_0^t \int_{u(s)>r} f_\ve(u^\ve) (u^\ve_{xxx})^2 dx ds \leq \|u_x^\ve(0, \cdot)\|_2^2.
\end{equation}
Since $u^\ve(0) \to u_0$ in $H^1(\T)$, then $u_x^\ve(t)$ is weakly compact in $L^2(0,L)$ for any $t \in [0,T].$ Then for any $t \in [0,T]$ there exists $\ve_k(t) \to 0$ as $k \to \infty$ such that for any $\f \in C^\infty (\T)$ and some $\xi \in L^2(\T)$ we have
\[
\int_0^L u_x^{\ve_k} (t,x) \f(x) dx \to \int_0^L \xi(x) \f(x) dx, \ k \to \infty.
\]
On the other hand, using the uniform convergence $u^\ve \to u$ in $C(Q_T)$ and Dominated Convergence Theorem, we have
\[
\int_0^L u_x^{\ve_k}(t,x) \f(x) dx  = -\int_0^L u^{\ve_k}(t,x) \f^{'}(x) dx \to -\int_0^L u(t,x) \f^{'}(x) dx = \int_0^L u_x(t,x) \f(x) dx
\]
since $u \in H^1(\T)$. Thus $u_x^{\ve_k} \to^w u_x, \ k \to \infty$ in $L^2(\T)$. By weak lower semi-continuity of the norm,
\begin{equation}\label{liminf}
    \liminf_{k \to \infty} \|u_x^{\ve_k(t)}(t)\|_2^2 \geq \|u_x(t)\|_2^2.
\end{equation}
Fix $t \in [0,T]$ and denote $\ve_k(t) = \ve_k$. Due to uniform convergence of $u$ and all of its derivatives (up to the fourth order) on any compact subset of $P$ we deduce
\begin{equation}\label{3.16}
    \|u_x(t)\|_2^2 + 2 \int_0^t \int_{u(s)>r} u^2(s) (u_{xxx}(s))^2 dx ds \leq \|u_x(0, \cdot)\|_2^2. 
\end{equation}
It follows from \eqref{3.16} that for arbitrary $r>0$ \eqref{pth_est} is satisfied for $p=2$. The estimate \eqref{pth_est} follows along the lines of \cite{Gess}, since
\begin{multline*}
 \| \d_x u(t, \cdot)\|_2^p + \int_0^t \|\d_x u(s, \cdot)\|_2^{p-2} \int_{u(s, \cdot) > 0} u^2(s,x) (\d_x^3 u(s,x))^2 dx ds \\
 \leq \sup_{s \in [0,t]} \|\d_x u(s)\|_2^{p-2} \left(\|\d_x u(t)\|^2_2  + \int_0^t \int_{u(s, \cdot) > 0} u^2(s,x) (\d_x^3 u(s,x))^2 dx ds \right).
\end{multline*}
\end{proof}
\section{Stochastic dynamics.}\label{Sec4}
In this section we consider the stochastic contribution to \eqref{1.1} on $[0,\delta) \times \T$ for some $0<\delta \leq T$:
\begin{equation}\label{eq:stoch}
\begin{cases}
    dw = \frac{1}{2} \sum_{k \in \Z} \lambda_k^2 \d_x(\Psi_k \d_x(\Psi_k w)) dt  + \sum_{k \in \Z}\lambda_k (\d_x(\Psi_k w)) d \beta_k(t) + \sum_{k \in \Z} \gamma_k \Psi_k f(w)) d\beta_1^k, \\
    w(0) = w_0.
\end{cases}    
\end{equation}
\begin{definition}\label{Def4.1}
    Let $w_0 = w_0(x, \omega) \in H^1(\T)$ almost surely. A triple consisting of a filtered probability space $([0,1], \tilde{\F}, (\tilde{\F}_{t})_{t \in [0,\delta)}, \mathbf{P})$, where $(\tilde{\F}_{t})_{t \in [0,\delta)}$ is a complete and right-continuous filtration, an  $(\tilde{\F}_t)$ - adapted $H_w^1(\T)$- valued process $\ti{W}(t)$ and mutually independent standard real-valued $(\tilde{\F}_t)$ - Wiener processes $\ti{\beta}^k(t)$ and $\ti{\beta}_1^k(t)$ is called the solution of \eqref{eq:stoch} if 
    \begin{align}\label{eq:stoch:weak}
        & & (\ti{w}(t,\cdot), \f)_2  = (w_0, \f)_2 - \frac{1}{2} \sum_{k \in \Z} \lambda_k^2 \int_0^t (\Psi_k \d_x(\Psi_k \ti{w}(s, \cdot)), \d_x\f)_2 d s \\
        \nonumber & & - \sum_{k \in \Z} \lambda_k \int_0^t (\Psi_k \ti{w}(s, \cdot)), \d_x \f)_2 d \ti{\beta}_k(s) + \sum_{k \in \Z} \gamma_k \int_0^t (\Psi_k f(\ti{w}(s, \cdot)), \f)_2 d \ti{\beta}_1^k(s)
    \end{align}
    is satisfied for every $\f \in C^\infty(\T)$ and $t \in [0, \delta)$ $\P = \lambda_{[0,1]}$ almost surely. 
\end{definition}
\begin{theorem}\label{Thm4.1}
    For any $p \in [2, \infty)$ and for any $w_0 \in L^p([0,1], \F_0, \P, H^1(\T))$, $w_0 \geq 0$ a.s., there exists a non-negative solution $w(t,x,\omega)$ of \eqref{eq:stoch:weak} with the initial condition $w_0$, and satisfies the a-priori estimates 
    \begin{equation}\label{eq:apr:est1}
    \E \sup_{t \in [0,\delta)} \|w(t, \cdot)\|_{1,2}^p \leq C_1 \E \|w_0\|_{1,2}^p;
    \end{equation}
    \begin{equation}\label{eq:apr:est2}
    \lim_{t \to \delta} \sup \E \|\d_x w(t, \cdot)\|_2^p \leq e^{C_2 \delta} \left(\E \|\d_x w_0\|_2^p + C_3 \delta \E \left(\int_0^L w_0 dx\right)^p\right),
    \end{equation}
    where $C_1, C_2, C_3 < \infty$ are independent of $\delta, w$ and $w_0$.
\end{theorem}
\begin{proof}
    {\bf Step 1: Regularization of the problem.} For any $\ve \in (0,1]$ we consider the viscous regularization of the problem \eqref{eq:stoch}:
    \begin{equation}\label{eq:stoch:reg}
dw^\ve = \left(\frac{1}{2} \sum_{k \in \Z} \lambda_k^2 \d_x(\Psi_k \d_x(\Psi_k w^\ve)) + \ve \d_x^2 w^\ve\right) dt  + \sum_{k \in \Z}\lambda_k (\d_x(\Psi_k w^\ve)) d\beta^k +
    \sum_{k \in \Z}\gamma_k \Psi_k f(w^\ve)) d\beta_1^k, 
\end{equation}
with the initial condition $w^\ve(0) = w_0$. Denote $A^\ve: H^2(\T) \to L^2(\T)$ to be
\begin{equation}\label{A_eps}
    A^\ve w := \frac{1}{2} \sum_{k \in \Z} \lambda_k^2 \d_x(\Psi_k \d_x(\Psi_k w^\ve)) + \ve \d_x^2 w^\ve,
\end{equation}
and the diagonal Hilbert-Schmidt-Valued operator 
\[
B: H^2(\T) \to L^2(H^2(\T) \times H^2(\T); H^1(\T)):=L^2
\]
is given by 
\[
B(w) \begin{pmatrix}
    u\\
    v
\end{pmatrix}
 = (B_1(w), B_2(w)) \begin{pmatrix}
    u\\
    v
\end{pmatrix} 
:= B_1(w) u + B_2(w) v,
\]
where 
\[
B_1(w) u  = \sum_{k \in \Z} \lambda_k(u, \Psi_k)_{2,2} (\d_x(\Psi_k w))
\]
and
\[
B_2(w) v  = \sum_{k \in \Z} \gamma_k(v, \Psi_k)_{2,2} \Psi_k f(w)
\]
Since $\begin{pmatrix}
    \Psi_k\\
    0
\end{pmatrix} $ and $\begin{pmatrix}
    0 \\
    \Psi_k
\end{pmatrix} $ form the basis in $H^2(\T) \times H^2(\T)$, then
\begin{multline*}
    \|B(w)\|^2_{L^2} := \|B_1(w) \Psi_k\|_{H^1(\T)}^2  +\|B_2(w) \Psi_k\|_{H^1(\T)}^2  \\
    = \sum_{k \in Z} \lambda_k^2 \int_0^L \left[(\d_x(\Psi_k w))^2 + (\d_x^2(\Psi_k w))^2\right] \, dx \\ + \sum_{k \in Z} \gamma_k^2 \int_0^L \left[(\Psi_k f(w))^2 + (\d_x(\Psi_k f(w)))^2\right] \, dx.
\end{multline*}

This way the equation \eqref{eq:stoch} can be re-written in the abstract form
\begin{equation}\label{eq:stoch:abs}
    d w^\ve = A^\ve w^\ve dt + (B w^\ve) d \overline{W}_{H^2(\T)},
\end{equation}
where
\[
\overline{W}_{H^2(\T)}:= (W, W_1) := \left(\sum_{k \in \Z} \beta^k \Psi_k, \sum_{k \in \Z} \beta_1^k \Psi_k \right)
\]
is a cylindrical $\F_t$-Wiener process in $H^2(\T) \times H^2(\T)$.

\begin{proposition}\label{Prop4.1}
 For any $p \in [2, \infty)$ and for any initial condition $w_0 \in L^p([0,1], \F_0, \P, H^1(\T))$,  the equation \eqref{eq:stoch:reg} has a unique variational solution $w^\ve$ satisfying
 \begin{equation}\label{4.10}
 \E \left(\sup_{t \in [0,\delta)} \|w^\ve(t,\cdot)\|_{1,2}^p + \int_0^\delta \|w^\ve(t, \cdot)\|_{2,2}^2 dt\right)< \infty.
 \end{equation}
 Furthermore, 
 \begin{enumerate}
     \item \begin{equation}\label{4.11}
         \E \sup_{t \in [0,\delta)} \|w^\ve(t, \cdot)\|_{1,2}^2 \leq C_1 \E \|w_0\|_{1,2}^p,
     \end{equation}
     \item If $w_0 \geq 0$ a.s., then $w^\ve(t,x) \geq 0$ a.s.
     \item If $w^\ve(t,x) \geq 0$ a.s., then
     \begin{equation}\label{4.13}
\lim_{t \to \delta} \E \|\d_x w^\ve\|_2^p \leq e^{C_2 \delta}\left(\E \|\d_x w_0\|_2^p + C_3 \delta \E \left| \int_0^L w_0 dx \right|^p \right)
     \end{equation}
     where $C_1, C_2, C_3$ are independent of $\ve, \delta$
 and $w_0$.
 \end{enumerate}
\end{proposition}
\begin{proof}
The conclusion of proposition \ref{Prop4.1} follows from Theorem 5.1.3 in \cite{LiuRoch} under the following conditions for any $u,v$ and $w$ in $H^2(\T)$:\\

(H1) \ (Hemicontinuity) The map $\lambda \to \langle \langle A^\ve(u+\lambda v),w\rangle \rangle$ is continuous on $\R$\\

(H2') \ (Local monotonicity)
\[
 2  \langle \langle A^\ve(u-v), u-v\rangle \rangle + \|B(u) - B(v)\|^2_{\mathcal{L}_2(H^2(\T) \times H^2(\T), H^1(\T))} \leq \rho(v) \|u-v\|_{H^1(\T)}^2
\]
for some non-negative, mearurable, hemicontinuous and locally bounded function on $V$.

(H3) (Coercivity) There are constants $C_0 \in \R$, $\alpha>1$ and $\theta>0$ such that
\[
2 \langle \langle A^\ve(v), v\rangle \rangle + \|B(v)\|^2_{\mathcal{L}_2(H^2(\T)\times H^2(\T), H^1(\T))} \leq C_0\|v\|_{H^1(\T)}^2  - \theta \|v\|_{H^2(\T)}^\alpha
\]

(H4') (Growth)
\[
\|A(v)\|_{L^2(\T)}^{\frac{\alpha}{\alpha-1}} \leq C_0\|v\|^\alpha_{H^2(\T)} (1+\|v\|_{H^1(\T)}^\beta)
\]
   Note that the above conditions make use of the Gelfand triple $V \subset H \subset V^*$, where $V^* = L^2(\T), \ H = H^1(\T)$ and $V = H^2(\T)$. 
   
   The condition (H1) (hemicontinuity) is straightforward.
   
   Next, we check the condition (H2'). We have
    \begin{eqnarray}\label{identity3}
\nonumber & 2 & \langle \langle A^\ve(u-v), u-v\rangle \rangle + \|B(u) - B(v)\|^2_{L^2(H^2(\T), H^1(\T))} = 2 \langle A^\ve(u-v), u-v \rangle \\
\nonumber & + &  2 \langle \d_x A^\ve(u-v), \d_x(u-v) \rangle + \|B(u) - B(v)\|^2_{\mathcal{L}_2(H^2(\T) \times H^2(\T); L^2(\T))} \\
\nonumber & + &  \|B(u) - B(v)\|^2_{\mathcal{L}_2(H^2(\T) \times H^2(\T); \dot{H}^1(\T))} \\
\nonumber & = & 2 \langle A^\ve(u-v), u-v \rangle + 2 \langle \d_x A^\ve(u-v), \d_x(u-v) \rangle \\
\nonumber & + & \sum_{k} \lambda_k^2 \int_0^L (\d_x(\Psi_k(u-v)))^2 \, dx +  \sum_{k} \lambda_k^2 \int_0^L (\d^2_x(\Psi_k(u-v)))^2 \, dx \\
& + & \sum_{k} \gamma_k^2 \int_0^L (\Psi_k(f(u)-f(v)))^2 \, dx + \sum_{k} \gamma_k^2 \int_0^L (\d_x(\Psi_k(f(u)-f(v))))^2 \, dx \\
\nonumber & := &  J_1 + J_2 + J_3 + J_4 + J_5 + J_6.
    \end{eqnarray}
    Let us analyze every term in \eqref{identity3} separately. Integrating by parts, we have
    
 \begin{eqnarray}\label{J1} 
 & & J_1 := -  \sum_k \lambda_k^2 \int_0^L \Psi_k^2 (\d_x(u-v)^2) dx - \frac{1}{4} \sum_k \lambda_k^2 \int_0^L (\d_x \Psi_k^2) (\d_x(u-v))^2 dx \\
\nonumber & - & 2\ve \int_0^L(\d_x(u-v))^2 dx = - \sum_k \lambda_k^2 \int_0^L \Psi_k^2 (\partial_x(u-v))^2\, dx + \frac{1}{4} \sum_k  \int_0^L (\d_x^2 \Psi_k^2)(u-v)^2 \, dx \\
 \nonumber & - & 2 \ve \int_0^L (\d_x(u-v))^2 dx \leq -\sum_k \lambda_k^2 \int_0^L \Psi_k^2 (\d_x(u-v)^2) dx + C \|u-v\|_2^2,
 \end{eqnarray}
for some $C>0$, independent of $\ve$. Here we used \eqref{1.3} and \eqref{1.5}. Next, following \cite{Gess}, we have
\begin{eqnarray}\label{J2} 
 & & J_2 := -  \sum_k \lambda_k^2 \int_0^L \Psi_k^2 (\d_x^2(u-v)^2) dx + \frac{5}{4} \sum_k \lambda_k^2 \int_0^L (\d_x^2 \Psi_k)^2 (\d_x(u-v))^2 dx \\
\nonumber & - & \frac{1}{4}  \sum_k \lambda_k^2 \int_0^L (\d_x^4 \Psi_k^2) (\d_x(u-v))^2 dx - 2 \ve \int_0^L (\d_x^2(u-v))^2 \, dx \\
 \nonumber & \leq & -  \sum_k \lambda_k^2 \int_0^L \Psi_k^2 (\d_x^2(u-v)^2) dx + C \|u-v\|_{1,2}^2  -   2 \ve \int_0^L (\d_x^2(u-v))^2 \, dx,
 \end{eqnarray}
    where, once again, $C>0$ is independent of $\ve$. 
  \begin{eqnarray}\label{J3} 
\nonumber & & J_3 = \sum_k \lambda_k^2 \int_0^L [(\d_x \Psi_k)(u-v) \Psi_k \d_x(u-v)]^2\, dx \\
&  & \leq 2 \left( \sum_k \lambda_k^2 \int_0^L \Psi^2_k (\d_x (u-v))^2 \,dx  -  \sum_k \lambda_k^2 \int_0^L  \Psi_k(\d_x^2 \Psi_k)(u-v)^2\, dx \right).
  \end{eqnarray}
  We proceed with the estimate for $J_4$:
   \begin{eqnarray}\label{J4} 
\nonumber & & J_4 = \sum_k \lambda_k^2 \int_0^L \Psi_k^2 (\d_x^2(u-v))^2  \, dx + \sum_k \lambda_k^2  \int_0^L (2 (\d_x \Psi_k)^2 - 4 \Psi_k (\d_k^2 \Psi_K))(\d_x(u-v))^2 \, dx  \\
 & + & \sum_k \lambda_k^2  \int_0^L \Psi_k (\d_x^4 \Psi_k)(u-v)^2 \, dx \leq \sum_k \lambda_k^2  \int_0^L \Psi_k^2 (\d_x^2(u-v))^2 \, dx + C\|u-v\|_{1,2}^2,
\end{eqnarray}
where $C>0$ is independent of $\ve$. 

The estimate of the fifth term uses the Lipschitz condition:
\begin{equation}\label{J5}
J_5 \leq  \sum_k \gamma_k^2 \int_0^L (|\Psi_k| K |u-v|)^2 \, dx \leq C \|u-v\|_2^2,
\end{equation}
where, once again, $C$ is independent of $\ve$. In order to estimate $J_6$, consider
 \begin{multline}\label{aux1J6} 
     \int_0^L \left(\d_x(\Psi_k(f(u)-f(v)))\right)^2 \, dx = \int_0^L\left[\d_x\Psi_k \, (f(u)-f(v)) + \Psi_k(f'(u) u_x - f'(v) v_x)\right]^2 \, dx\\
     \leq C \|u-v\|_2^2 + C \int_0^L |f'(u) u_x - f'(v) v_x|^2 \, dx.
 \end{multline}
 Since $f'(u)$ is both Lipschitz and bounded, we have
 \[
|f'(u) u_x - f'(v) v_x| \leq K|u_x - v_x| + K(u-v)|v_x|. 
 \]
Taking into account that $v_x$ is absolutely continuous due to the embedding $H^1 \subset C^{0,1/2}$, we have $\sup_{x \in \T}|v_x| \leq C_1 \|v\|_{2,2}$. Therefore,
\begin{equation}\label{aux2J6}
    \int_0^L |f'(u) u_x - f'(v) v_x| \, dx \leq C\|u-v\|_{1,2}^2 + C \|u-v\|_2^2 \|v\|_{2,2}^2.
\end{equation}
In view of \eqref{aux1J6} and \eqref{aux2J6}, we deduce
\begin{equation}\label{J6}
J_6 \leq C\|u-v\|_{1,2}^2 + C\|u-v\|_2^2 \|v\|_{2,2}^2.
\end{equation}
Combining the estimates for $J_1 - J_6$, we get
\begin{equation}\label{CondH2}
2 \langle \langle A^\ve(u-v), u-v \rangle \rangle + \|B(u) - B(v)\|^2_{L^2} \leq C (1 + \|v\|_{2,2}^2
)\|u-v\|_{1,2}^2 - 2 \ve \|u-v\|_{2,2}^2  
\end{equation}
which implies that the condition (H2') \, (local monotonicity) holds. Note that the constant $C>0$ in \eqref{CondH2} is independent of $u,v$ and $\ve$. 

Let us now check the coercivity condition (H3) from Theorem 5.1.3 \cite{LiuRoch}. In a similar way as before, we get
\[
2 \langle \langle A^\ve u, u\rangle \rangle + \|B_1(u)\|^2_{L^2(H^2(\T), H^1(\T))} \leq C \|u\|^2_{1,2} - 2\ve \|u\|^2_{2,2}
\]
for some constant $C>0$ independent of $u \in H^2(\T)$ and $\ve$. Let us proceed with the estimate for $B_2$:
\[
\|B_2(u)\|^2_{L^2(H^2(\T), H^1(\T))} = \sum_k \gamma_k^2 \int_0^L (\Psi_k f(u))^2 \, dx + \sum_k \gamma_k^2 \int_0^L (\d_x(\Psi_k f(u)))^2 \, dx := I_1 + I_2.
\]
We have 
\begin{equation}\label{I1}
    I_1 \leq C \|u\|_2^2
\end{equation}
which follows from the linear growth condition on $f(u)$. As for $I_2$,
\begin{equation}\label{I2}
    I_2 = \sum_k \gamma_k^2 \int_0^L [\d_x (\Psi_k) f(u) + \Psi_k f'(u) u_x]^2 \, dx \leq C \|u\|_2^2  + C\|u_x\|_2^2 = C \|u\|_{1,2}^2,
\end{equation}
    where, once again, we use the fact that $f'(u)$ is bounded. Thus, combining  the estimates \eqref{I1}-\eqref{I2}, we conclude 
    \begin{equation}\label{4.27}
        2 \langle \langle A^\ve u, u\rangle \rangle + \|B(u)\|^2_{\mathcal{L}_2(H^2(\T)\times H^2(\T), H^1(\T))} \leq C \|u\|^2_{1,2} - 2\ve \|u\|^2_{2,2},
\end{equation}
hence the condition (H3) holds. \\

    Finally, let us check the condition (H4'). For any $u \in H^2(\T)$ and $\f \in C^{\infty}(\T)$, integrating by parts we have
    \begin{multline*}
|\langle A^\ve u, \f\rangle| = \frac{1}{2} \sum_{k} \lambda_k^2 \left| \int_0^L \Psi_k(\d_x (\Psi_k u))(\d_x \f) dx \right| + \ve \left|\int_0^L (\d_x u ) (\d_x \f) \, dx\right|  \\
+ \frac{1}{2} \sum_k \lambda_k^2 \left|\int_0^L \left(\d_x (\Psi_k(\d_x(\Psi_k u)))\right)(\d_x^2 \f) \, dx \right| + \ve \left|\int_0^L (\d_x^2 u )(\d_x^2 \f) \, dx \right| \leq C \|u\|_{2,2} \|\f\|_{2,2},
    \end{multline*}
for some $C>0$ and $\ve<1$. We are now in position to apply Banach-Steinhaus theorem to obtain
\[
\|A^\ve u\|_{L^2(\T)} \leq C \|u\|_{2,2}.
\]
Therefore, the conditions of Theorem 5.1.3 \cite{LiuRoch} hold, and the regularized problem \eqref{eq:stoch:reg} admits the unique solution for all $\ve \in (0,1]$. 
We now proceed with the estimate \eqref{4.11}:
\begin{multline}\label{4.36}
  \|w^\ve(t)\|_{1,2}^2 - \|w_0\|_{1,2}^2 = 2 \sum_k \lambda_k \int_0^t(\d_x (\Psi_k w^\ve(s)), w^\ve(s))_{1,2} d \beta^k(s) \\
  + 2 \sum_k \gamma_k \int_0^t (\Psi_k f(w^\ve(s)), w^\ve(s))_{1,2} d \beta_1^k(s) + \int_0^t (2 \langle \langle A^\ve w^\ve(s), w^\ve(s) \rangle \rangle + \|B w^\ve(s)\|_{L^2}^2) \, ds
\end{multline}
for all $t \in [0,\delta)$ almost surely. Furthermore, for $p \geq 4$, using Ito's formula for $\|u\|_{1,2}^p$, we have
\begin{eqnarray}\label{4.37}
  \nonumber  &  &\|w^\ve(t)\|_{1,2}^p   -  \|w_0\|_{1,2}^p = p \sum_k \lambda_k \int_0^t \|w^\ve(s)\|_{1,2}^{p-2}(\d_x(\Psi_k w^\ve(s)), w^\ve(s))_{1,2} \, d \beta^k(s) \\
  \nonumber & + &  p \sum_k \gamma_k \int_0^t \|w^\ve(s)\|_{1,2}^{p-2}(\Psi_k f(w^\ve(s)), w^\ve(s))_{1,2} d\beta_1^k(s)  \\
   \nonumber & + & \frac{p}{2} \int_0^t \|w^\ve(s)\|_{1,2}^{p-2} (2 \langle \langle A^\ve w^\ve(s), w^\ve(s) \rangle \rangle + \|B(w^\ve(s))\|_{L_2}^2) \, ds \\
  \nonumber & + & \frac{p(p-2)}{2} \sum_k \lambda_k^2 \int_0^t \|w^\ve(s)\|_{1,2}^{p-4}(\d_x(\Psi_k w^\ve(s)),w^\ve(s))^2_{1,2} \, ds \\
   & + & \frac{p(p-2)}{2} \sum_{k} \gamma_k^2 \int_0^t \|w^\ve(s)\|_{1,2}^{p-4} (\Psi_k f(w^\ve(s)), w^\ve(s))^2_{1,2} \, ds
\end{eqnarray}
where we used the definition of the norm $\|B(u)\|_{L_2}$. Next, using Burkholder-Davis-Gundy inequality:
\begin{eqnarray} \label{4.38}
  \nonumber  &  & \E \sup_{t' \in [0,t)} \left| \int_0^{t'} \|w^\ve(s)\|_{1,2}^{p-2} (\d_x(\Psi_k w^\ve(s)),w^\ve)_{1,2} \, d \beta^k(s) \right| \\
   \nonumber & \leq  & 3 \E \left(\int_0^t \|w^\ve(s)\|_{1,2}^{2p-4} \d_x (\Psi_k w^\ve(s), w^\ve(s))_{1,2}^2 \, ds\right)^{\frac{1}{2}} \\
   \nonumber & \leq  & C \E \left(\int_0^t \|w^\ve(s)\|_{1,2}^{2p} \right)^{\frac{1}{2}} \leq C \E \sqrt{\sup_{s \in [0,t]}\|w^\ve(s)\|_{1,2}^p \, \int_0^t \|w^\ve(s) \|_{1,2}^p \, ds} \\
& \leq &  \nu  \E \sup_{s \in [0,t]}\|w^\ve(s)\|_{1,2}^p + \frac{C}{\nu} \E  \int_0^t \|w^\ve(s)\|_{1,2}^p \, ds,
\end{eqnarray}
where $\nu>0$ can be chosen arbitrarily small, and $C$ is independent of $\nu$. Note that the expected values in \eqref{4.38} are finite due to \eqref{4.10}.\\

In a similar way, we get
\begin{multline}\label{4.39}
\E \sup_{\tau \in [0,t]} \left| \int_0^{\tau} \|w^\ve(s)\|_{1,2}^{p-2} (\Psi_k f(w^\ve(s)), w^\ve(s))_{1,2} d \beta_1^k(s)\right| \\
\leq 3 \E \left(\int_0^t \|w^\ve(s)\|_{1,2}^{2p-4} (\Psi_k f(w^\ve(s)), w^\ve(s))_{1,2} \, ds \right)^{\frac{1}{2}}.
\end{multline}
But
\begin{multline}\label{4.40}
(\Psi_k f(w^\ve(s)), w^\ve(s))_{1,2} = \int_0^L \Psi_k f(w^\ve(s)) w^\ve(s) \, ds \\
+ \int_0^L  \d_x (\Psi_k f(w^\ve(s))) \d_x w^\ve(s) \, ds := I_1 + I_2.
\end{multline}
The estimate for $I_1$ follows from the conditions imposed on $f$, namely
\[
|I_1| \leq C \|w^\ve(s)\|_2^2.
\]
Similarly,
\[
|I_2| \leq \left|\int_0^L [\d_x (\Psi_k) f(w^\ve(s)) + \Psi_k f'(w^\ve(s)) \d_x(w^\ve(s))]  \d_x w^\ve(s) \, ds\right| \leq C \|w^\ve \|_{1,2}^2.
\]
Thus, the expression in \eqref{4.39} is estimated with
\begin{equation}\label{4.42}
\E \sup_{\tau \in [0,t]} \left| \int_0^\tau \|w^\ve(s)\|_{1,2}^{p-2}(\Psi_k f(w^\ve(s)), w^\ve(s))_{1,2} d \beta_1^k(s) \E \|w^\ve(s)\|\right| \leq C \int_0^t \E \|w^\ve(s)\|_{1,2}^2 \, ds
\end{equation}
The third term in \eqref{4.37}, using the same approach as we used in verification of coercivity (H3), can be bound as follows:
\begin{equation}\label{4.43}
 \frac{p}{2} \int_0^t \|w^\ve(s)\|_{1,2}^{p-2} (2 \langle \langle A^\ve w^\ve(s), w^\ve(s) \rangle \rangle + \|B(w^\ve(s))\|_{L_2}^2) \, ds \leq \frac{pc}{2} \int_0^t \E \|w_0^\ve(s)\|_{1,2}^p \, ds.
\end{equation}
Furthermore, using the same estimates, the last two terms in \eqref{4.37} can be estimated the same way as in \eqref{4.43}. Hence, combining \eqref{4.37}- \eqref{4.39}, \eqref{4.42} and \eqref{4.43}, for sufficiently small $\nu$ we have
\begin{equation}\label{4.44}
\E \sup_{s \in [0,t]} \|w^\ve(s)\|_{1,2}^p \leq C\left(\E \|w_0\|_{1,2}^p + \int_0^t \E \sup_{\tau \in [0,s]} \|w^\ve(\tau)\|_{1,2}^p\right),
\end{equation}
where $C>0$ only depends on $L$. The bound \eqref{4.11} now follows from Gronwall's inequality for $p=2$ and for $p \geq 4$. If $p \in (2,4)$, the bound follows from Young inequality and the interpolation inequality
\[
\|u\|_{\tilde{p}}^{\tilde{p}} \leq \|u\|_{p_0}^{p_0(1-\theta)} \|u\|_{p_1}^{p_1 \theta}, \ \text{ where } \tilde{p} = (1-\theta) p_0 + \theta p_1, \ \theta \in (0,1),
\]
taking $p_0 = 2$ and $p_1 = 4$. \\
Let us proceed with the proof of the second statement of Proposition \ref{Prop4.1}. In order to simplify the notation, in this part of the proof we are going to omit the subscript  $\ve$. Introduce
\[
u^+ = 
\begin{cases}
u, \ u \geq 0;\\
0, \ u < 0,
\end{cases}
\ \text{ and } \ 
u^- = 
\begin{cases}
u, \ u < 0;\\
0, \ u \geq 0
\end{cases}
\]
This way $u = u^+ + u^-$. The function $z = -u$ satisfies

\begin{multline}
dz = \left(\frac{1}{2} \sum_k \lambda_k^2 \d_x(\Psi_k \d_x(\Psi_k z)) + \ve \d_x^2 z \right) \, dt + \sum_{k} \lambda_k \d_x (\Psi_k z) \, d \beta^k(t) \\
    - \sum_{k \in \mathbb{Z}} \gamma_k \Psi_k f(-z) d \beta_1^k(t),
\end{multline}
\\
with $z(0) = -u_0 := z_0.$ Let  $\f(y)=(y^+)^2$ for any $y \in \R$. Now, let $\psi(y)$ be a $C^{\infty}$ function, such that
\[
\psi(y) = 
\begin{cases}
    0, \ y \in (-\infty, 1];\\
    1, \ y \in [2, \infty).
\end{cases}
\]
Set $\f_n(y) = y^2 \psi(ny)$. Then 
\begin{equation}\label{zirochka}
\lim_{n \to \infty} \f_n(y) = \f(y)
\end{equation}
uniformly in $y \in \R$. Furthermore, for any $y \in \R$,
\begin{equation}\label{4.45}
\f'_n(y) \to 2 y^+ \text{ and } \f_n^{''}(y) \to 2 \chi_{y>0}, \ n\to \infty.
\end{equation}
Furthermore,
\begin{equation}\label{4.46}
0 \leq \f_n(y) \leq \f(y), \ 0 \leq \f_n^{'}(y) \leq C y, \ \|\f_n^{''}(y)\| \leq C,
\end{equation}
for all $y \geq 0$ and  $n \geq 1$. Next, we apply Ito's formula to $\int_0^L \f_n(z(t)) dx$:
\begin{eqnarray}
   \nonumber  &  & \int_0^L \f_n(z(t)) dx = \int_0^L \f_n(z_0(x)) dx + \sum_k \lambda_k \int_0^t \int_0^L \f^{'}_n(z(s))  \d_x (\Psi_k z(s))dx d \beta^k(s) \\
    \nonumber  &  &  - \sum_k \gamma_k \int_0^t \int_0^L \f^{'}_n(z(s)) \Psi_k f(-z) dx d \beta_1^k(s) \\
    \nonumber  &  & + \int_0^t \int_0^L \f^{'}_n(z(s)) \left(\frac{1}{2} \sum_k \lambda_k^2 \d_x(\Psi_k \d_x( \Psi_k z(s))) + \ve \d_x^2 z \right) \, dx ds \\
    \nonumber  &  & + \frac{1}{2}  \sum_k \lambda_k^2 \int_0^t \int_0^L \f^{''}_n(z(s))(\d_x(\Psi_k z(s)))^2 \, dx ds \\
    &  & + \frac{1}{2}  \sum_k \gamma_k^2 \int_0^t \int_0^L \f^{''}_n(z(s))(\Psi_k f(-z(s)))^2 \, dx ds.
\end{eqnarray}
Integrating by parts, we have
\begin{eqnarray}
 & &\int_0^L \f^{'}_n(z(s)) \d_x(\Psi_k \d_x(\Psi_k z(s))) \, dx  = - \int_0^L \f^{''}_n (z(s)) z_x \Psi_k \d_x(\Psi_k z(s)) \, dx  \\
  &  & = - \int_0^L \f^{''}_n(z(s)) z_x \Psi_k (\d_x(\Psi_k) z + \Psi_k z_x) \, dx  = - \int_0^L \f^{''}(z(s)) z_x \Psi_k \d_x(\Psi_k) z \, dx  \\ 
   &  & - \int_0^L \f^{''}_n (z(s)) z_x^2  \Psi_k^ 2  \, dx.
\end{eqnarray}
Since
\begin{multline*}
    \int_0^L \f^{''}_n (z(s)) \left(\d_x(\Psi_k z)\right)^2 \, dx = \int_0^L \f^{''}_n (z(s)) \left(\d_x(\Psi_k)\right)^2 z^2 \, dx + \int_0^L \f^{''}_n (z(s)) \Psi_k^2  z_x^2 \, dx \\
    + 2 \int_0^L \f^{''}_n (z(s)) \, \d_x(\Psi_k) \Psi_k z z_x \, dx,
\end{multline*}
the expression \eqref{4.46} has the form
\begin{eqnarray}\label{4.48}
\nonumber & & \int_0^L \f_n(z(t)) \, dx = \int_0^L \f_n(z_0) \, dx + \sum_k \lambda_k \int_0^t \int_0^L \f^{'}_n (z(s)) \d_x (\Psi_k z(s)) \, dx \, d \beta^k(s) \\
 \nonumber & & - \sum_{k} \gamma_k \int_0^t \int_0^L \f^{'}_n (z(s)) \Psi_k f(-z) \, dx d\beta_1^k(s) - \frac{1}{2} \sum_{k} \lambda_k^2 \int_0^t \int_0^L \f^{''}_n (z) \Psi_k (\d_x \Psi_k) z z_x \, dx ds \\
  \nonumber & & - \frac{1}{2} \sum_{k} \lambda_k^2 \int_0^t \int_0^L \f^{''}_n (z) \Psi_k^2 z_x^2  \, dx ds - \ve \int_0^t \int_0^L \f^{''}_n (z) z_x^2 \, dx ds \\
  \nonumber & & + \frac{1}{2} \sum_{k} \lambda_k^2 \int_0^t \int_0^L \f^{''}_n (z)\left[(\d_x \Psi_k)^2 z^2 + 2 (\d_x \Psi_k) \Psi_k z z_x + \Psi_k^2 z_k^2\right]\, dx ds \\
   \nonumber & & + \frac{1}{2} \sum_{k} \gamma_k^2 \int_0^t \int_0^L \f^{''}_n (z)(\Psi_k f(-z(s)))^2\, dx ds  \\
   \nonumber & &  = \int_0^L \f_n(z_0) \, dx + \sum_k \lambda_k \int_0^t \int_0^L \f^{'}_n (z(s)) \d_x (\Psi_k z(s)) \, dx \, d \beta^k(s) \\
 \nonumber & & - \sum_{k} \gamma_k \int_0^t \int_0^L \f^{'}_n (z(s)) \Psi_k f(-z) \, dx d\beta_1^k(s) - \frac{1}{2} \sum_{k} \lambda_k^2 \int_0^t \int_0^L \f^{''}_n (z) \Psi_k (\d_x \Psi_k) z z_x \, dx ds \\
  \nonumber & &  - \ve \int_0^t \int_0^L \f^{''}_n (z) z_x^2 \, dx ds + \frac{1}{2} \sum_{k} \lambda_k^2 \int_0^t \int_0^L \f^{''}_n (z)\left[(\d_x \Psi_k)^2 z^2 + 2 (\d_x \Psi_k) \Psi_k z z_x \right]\, dx ds \\
    & & + \frac{1}{2} \sum_{k} \gamma_k^2 \int_0^t \int_0^L \f^{''}_n (z)(\Psi_k f(-z(s)))^2\, dx ds.
\end{eqnarray}
Taking the expected value in \eqref{4.48}, we may use \eqref{zirochka}, \eqref{4.46}, the dominated convergence theorem, and \eqref{4.11}, to pass to the limit as $n \to \infty$:
\begin{eqnarray}\label{4.49}
\nonumber & & \E \int_0^L (z^+(t,x))^2 \, dx = \E \int_0^L (z^+_0(x))^2 \, dx - 2 \ve \int_0^t \E \left(\int_0^L \chi_{\{z(s)>0\}} z_x^2 dx \right)  ds \\
\nonumber & & +  \sum_{k} \lambda_k^2 \int_0^t \E \left( \int_0^L \chi_{\{z(s)>0\}} (\d_x \Psi_k)^2 z^2(s) dx \right) \, ds  \\
\nonumber & & +  \sum_{k} \lambda_k^2 \int_0^t \E \left( \int_0^L \chi_{\{z(s)>0\}} (\d_x \Psi_k) \Psi_k z z_x dx \right) \, ds  \\
\nonumber & & +  \sum_{k} \gamma_k^2 \int_0^t \E \left( \int_0^L \chi_{\{z(s)>0\}} (\Psi_k f(-z))^2 dx \right) \, ds \\
\nonumber & & = \E \int_0^L (z^+_0(x))^2 \, dx - 2 \ve \int_0^t \E \left(\int_0^L (z^+_x)^2 dx \right)  ds \\
\nonumber & & +  \sum_{k} \lambda_k^2 \int_0^t \E \left( \int_0^L  (\d_x \Psi_k)^2 (z^+(s))^2 dx \right) \, ds  \\
\nonumber & & +  \sum_{k} \lambda_k^2 \int_0^t \E \left( \int_0^L  \Psi_k (\d_x \Psi_k) z^+ z^+_x dx \right) \, ds  \\
\nonumber & & +  \sum_{k} \gamma_k^2 \int_0^t \E \left( \int_0^L (\Psi_k f(-z^+))^2 dx \right) \, ds \\
\nonumber & & := \E \int_0^L (z^+_0(x))^2 \, dx - 2 \ve \int_0^t \E \left(\int_0^L (z^+_x)^2 dx \right)  ds + J_1 + J_2 + J_3.
\end{eqnarray}
Using \eqref{1.3} and \eqref{1.5}, we have
\begin{equation}\label{4.50}
    J_1 \leq C \int_0^t \E \int_0^L (z^+(s))^2 dx ds.
\end{equation}
In order to estimate $J_2$, integrating by parts, we have
\[
\int_0^L  \Psi_k (\d_x \Psi_k) z^+ z^+_x dx  = \frac{1}{4}\int_0^L   \d_x (\Psi_k^2) \, \d_x (z^+(s))^2  dx  = -\frac{1}{4} \int_0^L  \d_x^2 (\Psi_k^2) (z^+(s))^2 dx.
\]
This way
\begin{equation}\label{4.51}
    J_2 \leq C  \int_0^t \E \int_0^L (z^+(s))^2 dx ds.
\end{equation}
In order to estimate $J_3$, we make use of the conditions on $f(u)$. We have
\begin{equation}\label{4.52}
    J_3 \leq \sum_{k} \lambda_k^2 \int_0^t \E \left( \int_0^L K^2 \Psi_k^2 (z^+(s))^2 dx \right
) \, ds \leq C \int_0^t \E \int_0^L (z^+(s))^2 dx \, ds.
\end{equation}
It follows from \eqref{4.49}-\eqref{4.52} we obtain
\begin{multline}
     \E \int_0^L (z^+(t,x))^2 \, dx \leq \E \int_0^L (z_0^+(x))^2 \, dx - 2 \ve \int_0^t \E \int_0^L (z^+(s))^2 dx ds \\
     + C_1 \int_0^t \E \left(\int_0^L (z^+(s))^2 \, dx \right) ds.
\end{multline}
   Hence, by Gronwall's inequality, we have
   \[
 \int_0^L (z^+(t,x))^2 \, dx \leq e^{C_1 t} \E \int_0^L (z_0^+(x))^2 \, dx  = 0,
   \]
since $z_0^+ =0$ a.s. Therefore, $z(t,x) \leq 0$ a.s., and hence $u(t,x) \geq 0$ a.s. Thus the second statement of Proposition \ref{Prop4.1} follows.
In order to prove the third statement, we need the following Lemma:
\begin{lemma}
    Let $w^\ve(t,x)$ be a non-negative solution of \eqref{eq:stoch:reg}. Then 
    \begin{equation}\label{4.53}
        \E \left(\int_0^L w^\ve(t,x) dx\right)^p \leq e^{Ct}\E \left(\int_0^L w_0(x) dx\right)^p
    \end{equation}
    for some $C > 0$ and $p \geq 2$.
    \end{lemma}
    \begin{proof}
        Consider $F:L^2(\T) \to \R$ given by
        \[
        F(w) := \left(\int_0^L w(x) dx\right)^p.
        \]
        By Ito's formula,
        \begin{eqnarray}
 \nonumber & &\left(\int_0^L w^\ve(t) \, dx \right)^p - \left(\int_0^L w_0 \, dx \right)^p = \frac{1}{2} \sum_k \lambda_k^2  \int_0^t p \left(\int_0^L w^\ve(s) \, dx\right)^{p-1} \\
 \nonumber & & \times \int_0^L \d_x (\Psi_k \d_x (\Psi_k w^\ve(s))) \, dx \, ds + \int_0^t p \left(\int_0^L w^\ve(s) \, dx\right)^{p-1} \int_0^L \ve (\d_x^2 w^\ve(s)) \, dx \, ds \\
  \nonumber & & + \sum_k \lambda_k \int_0^t\left( p \left(\int_0^L w^\ve(s) \, dx\right)^{p-1} \int_0^L \d_x(\Psi_k w^\ve(s)) dx\right) \, d \beta^k(s)\\
 \nonumber & &  + \sum_k \gamma_k \int_0^t \left( p \left(\int_0^L w^\ve(s) \, dx\right)^{p-1} \int_0^L  \Psi_k f(w^\ve(s)) \, dx\right) d \beta_1^k(s)\\
 \nonumber & & + \frac{1}{2} \sum_k  \lambda_k^2 \int_0^t\left( p(p-1) \left(\int_0^L w^\ve(s) \, dx\right)^{p-2} \left(\int_0^L \d_x(\Psi_k w^\ve(s)) dx\right)^2\right) \, d s \\
 \nonumber & & + \frac{1}{2} \sum_k  \gamma_k^2 \int_0^t\left( p(p-1) \left(\int_0^L w^\ve(s) \, dx\right)^{p-2} \left(\int_0^L \Psi_k f(w^\ve(s)) dx\right)^2\right) \, d s
  \end{eqnarray}
  The first, the second, the third and the fifth terms in this expansion are zero due to the periodic boundary conditions. Thus, taking the expected value, we have
  \begin{multline}
\E \left(\int_0^L w^\ve(t) \, dx \right)^p = \E \left(\int_0^L w_0 \, dx \right)^p \\
+ \frac{1}{2} \sum_k  \gamma_k^2 \int_0^t\left( p(p-1) \E \left(\int_0^L w^\ve(s) \, dx\right)^{p-2} \left(\int_0^L \Psi_k f(w^\ve(s)) dx\right)^2\right) \, d s
\end{multline}
Since 
\[
\left|\int_0^L \Psi_k f(w^\ve(s)) \, dx \right| \leq C \int_0^L |w^\ve(s))| \, dx  =  C \int_0^L w^\ve(s)) \, dx
\]
due to the condition \eqref{1.5}, the linear growth condition on $f(u)$, and non-negativity of $w^\ve(s)$, we have
\[
\E \left(\int_0^L w^\ve(t) \, dx \right)^p \leq  \E \left(\int_0^L w_0 \, dx \right)^p + C \int_0^t \E \left(\int_0^L w^\ve(t) \, dx \right)^p \, ds,
\]
and thus \eqref{4.53} follows from Gronwall Lemma. This completes the proof of the Lemma.
    \end{proof}
    We now obtain the inequality \eqref{4.13}.  Using Ito's formula we get for $t \in [0, \delta)$
    \begin{eqnarray}
     \nonumber & & \|\d_x w^\ve(t)\|_2^2 -  \|\d_x w_0\|_2^2  = 2 \sum_{k \in \Z} \lambda_k \int_0^t \left(\d_x^2 (\Psi_k w^\ve(s)), \d_x w^\ve(s)\right)_2 \, d \beta^k(s) \\
     \nonumber & & + 2 \sum_k \gamma_k \int_0^t \left(\d_x (\Psi_k f(w^\ve(s))), \d_x w^\ve(s)\right) d \beta_1^k(s) \\
     \nonumber & & + \int_0^t \left(2 \langle \d_x A^\ve w^\ve(s), \d_x w^\ve(s) \rangle + \left\|B w^\ve(s)\right\|^2_{L_2(H^2(\T), \dot{H}^1(\T)}\right) \, ds
    \end{eqnarray}
    For $p \geq 4$, using Ito's formula one again, we have
     \begin{eqnarray}
     \nonumber & & \|\d_x w^\ve(t)\|_2^p -  \|\d_x w_0\|_2^p = p \sum_{k \in \Z} \lambda_k \int_0^t \|\d_x w^\ve\|_2^{p-2} \left(\d_x^2 (\Psi_k w^\ve(s)), \d_x w^\ve(s)\right)_2 \, d \beta^k(s)  \\
     \nonumber & & + p \sum_k \gamma_k \int_0^t \|\d_x w^\ve\|_2^{p-2} \left(\d_x (\Psi_k f(w^\ve(s))), \d_x w^\ve(s)\right) d \beta_1^k(s) \\
     \nonumber & & + 
     \frac{p}{2} \int_0^t \|\d_x w^\ve\|_2^{p-2} \left(2 \langle \d_x A^\ve w^\ve(s), \d_x w^\ve(s) \rangle + \left\|B w^\ve(s)\right\|^2_{L_2(H^2(\T), \dot{H}^1(\T)}\right) \, ds \\
     \nonumber & & + \frac{p(p-2)}{2} \sum_k \lambda_k^2 \int_0^t  \|\d_x w^\ve\|_2^{p-4} \left(\d_x^2 (\Psi_k w^\ve(s)), \d_x w^\ve(s)\right)_2^2 \, ds \\
     \nonumber & & + \frac{p(p-2)}{2} \sum_k \gamma_k^2 \int_0^t  \|\d_x w^\ve\|_2^{p-4} \left(\d_x (\Psi_k f(w^\ve(s))), \d_x w^\ve(s)\right)_2^2 \, ds 
    \end{eqnarray}
    Taking the expected value yields
    \begin{eqnarray}\label{4.54}
      & & \E \|\d_x w^\ve(t)\|_2^p -  \E \|\d_x w_0\|_2^p \\
 \nonumber & & 
     = \frac{p}{2} \E \int_0^t \|\d_x w^\ve\|_2^{p-2} \left(2 \langle \d_x A^\ve w^\ve(s), \d_x w^\ve(s) \rangle + \left\|B w^\ve(s)\right\|^2_{L_2(H^2(\T), \dot{H}^1(\T)}\right) \, ds \\
     \nonumber & & + \frac{p(p-2)}{2} \sum_k \lambda_k^2 \E \int_0^t  \|\d_x w^\ve\|_2^{p-4} \left(\d_x^2 (\Psi_k w^\ve(s)), \d_x w^\ve(s)\right)_2^2 \, ds \\
     \nonumber & & + \frac{p(p-2)}{2} \sum_k \gamma_k^2 \E \int_0^t  \|\d_x w^\ve\|_2^{p-4} \left(\d_x (\Psi_k f(w^\ve(s))), \d_x w^\ve(s)\right)_2^2 \, ds 
     \end{eqnarray}
Integrating by parts, 
\begin{multline}\label{4.55}
    \left|\left(\d_x^2(\Psi_k w^\ve(s)), \d_x w^\ve(s)\right)_2 \right| = \left|\int_0^L (\d_x^2 \Psi_k) w^\ve \d_x w^\ve \, dx + \frac{3}{2}  \int_0^L (\d_x \Psi_k)(\d_x w^\ve(s))^2 \, dx \right| \\
    \leq C (\|w^\ve(s)\|_2  \|w_x^\ve(s)\|_2 + \|w_x^\ve\|_2^2)  = C \|w^\ve(s)\|_{1,2} \|w_x^2(s)\|_2
\end{multline}
Similarly, 
\begin{eqnarray}\label{4.56} 
\nonumber & &(\d_x (\Psi_k f(w^\ve(s))), w^\ve_x(s))_2  = \int_0^L [\d_x (\Psi_k) f(w^\ve(s))w_x^\ve(s) + \Psi_k f'(w^\ve(s)) w_x^\ve(s)]\, dx \\
 & & \leq C \int_0^L w^\ve(s)  w^\ve_x(s) \, dx + C \int_0^L (w_x^\ve(s))^2 \, dx \leq C \|w^\ve(s)\|_{1,2} \|w_x^2(s)\|_2,
\end{eqnarray}
where we used \eqref{1.5}, nonnegativity of $w^\ve$, and the linear growth of $f(u)$.
Using  the inequality 
\[
2 \langle \d_x A^\ve u, \d_x u \rangle + \|B(u)\|^2_{\mathcal{L}_2(H^2(\T) \times H^2(\T), \dot{H}_1(\T))} \leq C \|u\|^2_{1,2} - 2 \ve \|\d_x u\|_{1,2}^2,
\]
which we obtained in \eqref{4.27}, as well as \eqref{4.55} and \eqref{4.56}, we get
\begin{multline}
  \E \|\d_x w^\ve(t)\|_2^p -  \E \|\d_x w_0\|_2^p \leq C \frac{p}{2} \E \int_0^t \|\d_x w^\ve\|_2^{p-2} \|w^\ve(s)\|_{1,2}^2\, ds \\
     + C \frac{p(p-2)}{2} \sum_k (\lambda_k^2 + \gamma_k^2) \E \int_0^t  \|\d_x w^\ve\|_2^{p-2} \|w^\ve(s)\|_{1,2}^2 \, ds.
\end{multline}
By Poincare inequality 
\[
\|w^\ve(s)\|_{1,2} \leq C\left(\|\d_x w^\ve(s)\|_2 + \left|\int_0^L w^\ve(s) \, dx \right| \right)
\]
we have
\begin{multline}
  \E \|\d_x w^\ve(t)\|_2^p -  \E \|\d_x w_0\|_2^p \leq C_1 \int_0^t \E \|\d_x w^\ve(s)\|_2^p \, ds + C_2 \int_0^t \E  \|w_x^\ve\|_2^{p-2} \left(\int_0^L w^\ve(s) \, dx \right)^2 \, ds
\end{multline}
Using Young's inequality,
\[
\E \|\d_x w^\ve(t)\|_2^p \leq  \E \|\d_x w^\ve(0)\|_2^p + C_3 \int_0^t \E \|\d_x w^\ve(s)\|_2^p \, ds + C_4 \int_0^t \E \left(\int_0^L w^\ve(s) \, dx\right)^p \, ds
\]
Hence using the estimate \eqref{4.53},  we have 
\[
\E \|\d_x w^\ve(t)\|_2^p \leq \E \|\d_x w^\ve(0)\|_2^p + C_3 \int_0^t \E \|\d_x w^\ve(s)\|_2^p \, ds + C_4  e^\delta \E \left(\int_0^L w^\ve_0 \, dx\right)^p \, \delta
\]
The conclusion of Proposition \ref{Prop4.1} now follows from Gronwall's inequality. 

\end{proof}
{\bf Step 2.} We now return to the proof of Theorem \ref{Thm4.1}. Proposition \ref{Prop4.1} implies that the initial value problem \eqref{eq:stoch:reg} has a unique solution $w^\ve(t)$ for any $\ve>0$. Thus, for the sake of proving Theorem \ref{Thm4.1}, one needs to pass to the limit in \eqref{eq:stoch:reg} as $\ve \to 0$ and to show that $w^\ve(t)$ converges to $w(t)$, which solves \eqref{eq:stoch}. We shall need the analog of Lemma 4.4 \cite{Gess}:
\begin{lemma}\label{lem4.2}
    For any $p \in [2, \infty), \nu > 0$ and $q \in [p, \infty)$ there exists a constant $C>0$ such that for all $\gamma \in (0,1)$ we have
    \[
w^\ve \in L^p(\Omega, [0,1], \mathcal{F}, \lambda_{[0,1]}; B_{q}^{\frac{\gamma}{2} - \nu,q}([0,\delta), B_q^{\frac{1}{2} - 2 \gamma, q}(\T)))
    \]
    and
    \begin{equation}\label{4.57}
\E \|w^\ve\|^p_{B_{q}^{\frac{\gamma}{2} - \nu,q}([0,\delta), B_q^{\frac{1}{2} - 2 \gamma, q}(\T))} \leq C \E \|w_0\|_{1,2}^p
    \end{equation}
with $C$ independent of $\ve$. 
\end{lemma}
The definition of Besov space $B_{q}^{s,p}(\Omega, X)$, with $\Omega \in \R^d$ and Banach space $X$ can be found in \cite{5,56}.
\begin{proof}
    For any $\ve>0$ and $\f \in C^\infty(\T)$ we have
    \begin{eqnarray}\label{4.58}
 & &  (w^\ve(t),\f)_2 = (w_0, \f) - \frac{1}{2} \sum_k \lambda_k^2 \int_0^t (\Psi_k \d_x(\Psi_k w^\ve(s)), \d_x \f)_2, \, ds - \ve \int_0^t (\d_x w^\ve(s), \d_x \f)_2 \, ds\\
 & & \nonumber - \sum_k \lambda_k \int_0^t(\Psi_k w^\ve(s),\d_x \f)_2 d \beta^k(s) + \sum_k \gamma_k \int_0^t (\Psi_k f(w^\ve(s)), \f)_2 \, d \beta_1^k(s)
    \end{eqnarray}
Introduce the following notation:
\[
w_1^\ve = \frac{1}{2} \sum_k \lambda_k^2 \int_0^t \d_x(\Psi_k \d_x(\Psi_k w^\ve(s))) \, ds - \ve \int_0^t \d_x^2 w^\ve(s) \, ds
\]
\[
w_2^\ve = \sum_k \lambda_k \int_0^t \d_x(\Psi_k w^\ve(s)) \, d \beta^k(s)
\]
\[
w_3^\ve = \sum_k \gamma_k \int_0^t \Psi_k f(w^\ve(s)) \, d \beta_1^k(s)
\]
Using \eqref{4.58}, for $0 \leq t_1 \leq t_2 < \delta$
\begin{multline*}
\|w_1^\ve(t_2) - w_1^\ve(t_1)\|^p_{H^{-1}(\T)} \leq C \left( \sum_k \lambda_k^2 \int_{t_1}^{t_2}(\|w^\ve(s)\|_{1,2} + \ve \|w^\ve(s)\|_{1,2})\right)^p \\
\leq C|t_2-t_1|^p \sup_{t \in [0,\delta)}\|w^\ve(s)\|_{1,2}^p
\end{multline*}
almost surely. Then from \eqref{4.10} we obtain
\begin{equation}\label{4.59}
\|w_1^\ve\|_{L^p([0,1], \mathcal{F}, \lambda_{[0,1]}; C^{1-}([0,\delta];H^{-1}(\T))} \leq C \E \|w_0\|_{1,2}^p
\end{equation}
We next estimate Sobolev-Slobodeckij norm \cite{20} (Lemma 2.1) with $\alpha = \frac{1}{2} - \nu < \frac{1}{2}$:
\begin{multline}
    \left(\E \|w_2^\ve\|^p_{W^{\frac{1}{2} - \nu, q}([0, \delta);L^2(\T))}\right)^q \leq \E \|w_2^\ve\|^q_{W^{\frac{1}{2} - \nu, q}([0, \delta);L^2(\T))} \\
    \leq C \E \int_0^\delta \left(\sum_{k \in \Z} \lambda_k^2 \|\d_x(\Psi_k w^\ve(s))\|_2^2\right)^\frac{q}{2} \, ds \leq C \delta \E \sup_{t \in [0,\delta)}\|w^\ve(t)\|_{1,2}^q \leq C \delta \E \|w_0\|_{1,2}^q
\end{multline}
Here we used \eqref{4.11}. An analogous estimate, based on the linear growth of $f(u)$, will hold for $w_3^\ve$. The rest of the proof of the Lemma follows the lines of Lemma 4.4 \cite{Gess}. 
\end{proof}
We now proceed with the proof of Theorem \ref{Thm4.1}. Let us choose the following parameters in \eqref{4.57}: $ \kappa \in [0,\frac{1}{2})$, $\nu \in \left(0, \frac{\kappa}{2} \right)$, $p = 2$, $q \in [2, \infty)$. Using the inequality \eqref{4.57}, and Markov's inequality, for $R>0$ we have
\[
\lambda_{[0,1]} \{\|w^\ve\|_{B_q^{\frac{\kappa}{2} - \nu, q}([0, \delta); B^{\frac{1}{2} - 2 \kappa,q}(\T))} > R\} \leq \frac{C^2}{R^2} \E \|w_0\|_{1,2}^p \to 0 \text{ as } R \to \infty.
\]
Thus the set 
\[
\{w: \|w^\ve\|_{B_q^{\frac{\kappa}{2} - \nu, q}([0, \delta); B^{\frac{1}{2} - 2 \kappa,q}(\T))} \leq R\}
\]
is compact in $BC^0([0,\delta) \times \T)$ when $\kappa< \frac{1}{4}$ and $q > \max\{\frac{2}{\kappa - 2\nu}, \frac{2}{1 - 4\kappa}\}$, \cite{Amann}, Th.4.4.
The set of measures 
\[
\mu_\ve(A) = \lambda_{[0,1]}\{\omega: w^\ve(t,x,\omega) \in A\}, \ A \subset BC^{0}([0,\delta) \times \T) 
\]
is tight in $BC^{0}([0,\delta) \times \T)$. Thus, by Skorokhod's Theorem \cite{Skor}, there exists a random variable $\tilde{w}^\ve$, and Wiener process $\tilde{\overline{W}}^\ve(t)$ on $([0,1], B([0,1], \lambda_{[0,1]})$ such that  $(\tilde{w}^\ve, \tilde{\overline{W}}^\ve) \sim (\w^\ve, \bar{W})$ and $\tilde{w}^\ve \to \tilde{w}$ a.s. in $BC^0([0,\delta) \times \T)$, $\tilde{\overline{W}}^\ve \to \bar{W}$ a.s. in $BC^0([0,\delta); H^2(\T))$ as $\ve \to 0$. Let us show that $\tilde{w}$ is the solution of the limiting equation \eqref{eq:stoch} with Wiener process $\tilde{W}$. Introduce the following real-valued processes:
\begin{eqnarray}
     & & \nonumber  \beta^k(t) = \lambda_k^{-1}(W(t),\Psi_k)_{2,2}, \\
     & & \nonumber  \beta^k_1(t) = \gamma_k^{-1}(W_1(t),\Psi_k)_{2,2}, \\
     & & \nonumber  \tilde{\beta}^k(t) = \lambda_k^{-1}(\tilde{W}(t),\Psi_k)_{2,2}, \\
     & & \nonumber  \tilde{\beta}^k_\ve(t) = \lambda_k^{-1}(\tilde{W}^\ve(t),\Psi_k)_{2,2}, \\
     & & \nonumber  \beta^k_{1,\ve}(t) = \gamma_k^{-1}(\tilde{W}_1^\ve(t),\Psi_k)_{2,2}, \\
     & & \nonumber  \beta^k(t) = \gamma_k^{-1}(\tilde{W}_1(t),\Psi_k)_{2,2}.
\end{eqnarray}
Additionally, we denote
\[
\bar{W}^\ve = (\tilde{W}^\ve, \tilde{W}_1^\ve), \ \bar{W} = (\tilde{W}, \tilde{W}_1)
\]
Similarly to Proposition 5.3 \cite{Gess} we can show that that the processes $\tilde{\beta}^k$ and $\tilde{\beta}^k_1$ are mutually independent standard real-valued $\tilde{F}_t$-Wiener processes. Here 
\[
\tilde{F}_t = \sigma((\tilde{w}(s), \tilde{\overline{W}}(s)), s \in [0,t])
\]
Since $\mathcal{L}(w^\ve, \bar{W}) = \mathcal{L}(\tilde{w}^\ve, \tilde{\overline{W}}^\ve)$, using \eqref{4.58} we have 
\begin{eqnarray}\label{4.60}
     & & \nonumber (\tilde{w}^\ve(t), \f) - (w_0,\f) = -\frac{1}{2} \sum_k \lambda_k^2 \int_0^t (\Psi_k \d_x (\Psi_k \tilde{w}^\ve(s)), \d_x \f)_2 \, ds - \ve \int_0^t (\d_x \tilde{w}^\ve(s), \d_x \f)_2 \, ds \\
     & &  - \sum_k \lambda_k \int_0^t (\Psi_k w^\ve(s), \d_x \f)_2 d \tilde{\beta}_\ve^k(s) + \sum_k \gamma_k \int_0^t (\Psi_k f(\tilde{w}^\ve(s)), \f)_2 d \tilde{\beta}_{1,\ve}^k(s).
\end{eqnarray}
Passing to the limit as $\ve \to 0$, we obtain the desired result. In view of the convergence $(\tilde{w}^\ve(t), \f) \to (\tilde{w},\f)_2, \ve \to 0$, we have
\begin{eqnarray}
     & & \nonumber \sum_k \lambda_k^2 \int_0^t (\Psi_k \d_x (\Psi_k \tilde{w}^\ve(s)), \d_x \f)_2 \, ds = - \sum_k \lambda_k^2 \int_0^t (\tilde{w}^\ve(s), \Psi_k \d_x (\Psi_k \d_x \f)_2)_2 \, ds \\
     & & \nonumber \to - \sum_k \lambda_k^2 \int_0^t (\tilde{w}(s), \Psi_k \d_x (\Psi_k \d_x \f)_2)_2 \, ds = \sum_k \lambda_k^2 \int_0^t (\Psi_k \d_x (\Psi_k \tilde{w}(s)), \d_x \f)_2 \, ds \text{ as }  \ve \to 0.
\end{eqnarray}
Integrating by parts, we have
\[
\ve \int_0^t \d_x \tilde{w}^\ve(s), \d_x \f) \, ds = - \ve \int_0^t (\tilde{w}^\ve(s), \d_x^2 \f)_2 \to 0, \ \ve \to 0,
\]
since there exists a random constant $C(\omega)>0$ such that
\[
\sup_{t \in [0,\delta), \ x \in \T} |\tilde{w}^\ve(t,x,\omega)| \leq C(\omega) \text{ a.s. }
\]
Let us show the convergence of the two last terms in \eqref{4.60}. To this end, we rewrite the in the form
\begin{multline}\label{4.63}
    \sum_k \left[\int_0^t \gamma_k (\Psi_k f(\tilde{w}^\ve(s)),\f)_2 d \tilde{\beta}_{1,\ve}^k - \lambda_k \int_0^t (\Psi_k w^\ve(s), \d_x \f)_2 d \tilde{\beta}_\ve^k(s)\right] 
    = (\tilde{w}^\ve(t), \f)_2 - (w_0,\f)_2  \\
    + \ve \int_0^t (\d_x (\tilde{w}^\ve(s)), \d_x \f)_2 \, ds + \frac{1}{2} \sum_k \lambda_k^2 \int_0^t (\Psi_k \d_x (\Psi_k \tilde{w}^\ve(s)),\d_x \f)_2 \, ds = \tilde{M}_\ve(t)
\end{multline}
Thus $\tilde{M}_\ve(t)$ is a martingale with respect to sigma-algebra $\tilde{F}_{t,\ve} = \sigma((\tilde{w}^\ve(s), \tilde{W}(s)), s \in [0,t])$. Similarly to Lemma 5.7 \cite{Gess} one can show that the quadratic variation of the process satisfies the following bound:
\begin{multline}
\langle \langle \tilde{M} \rangle \rangle_t = \sum_k \gamma_k^2 \int_0^t (\Psi_k f(\tilde{w}^\ve(s)), \f)_2^2 \, ds + \sum_k \lambda_k^2 \int_0^t (\Psi_k \tilde{w}^\ve(s), \d_x \f)_2^2 \, ds \\
\leq C \|\f\|_{1,2}^2 \int_0^t\|\tw^\ve(s)\|_2^2 \, ds.
\end{multline}
This way
\[
\E \left(\langle \langle \tilde{M} \rangle \rangle_t\right)^p \leq C t^p \|\f\|_{1,2}^{2p} \E \sup_{s \in [0,\delta)}(\|\tw^\ve(s)\|_{1,2}^{2p}) \leq C \delta^p \|\f\|_{1,2}^{2p} \E \|w_0\|_{1,2}^{2p} \text{ for } p \geq 1,
\]
where we used \eqref{4.11} and the fact that the distributions for $\tw^\ve$ and $w^\ve$ coincide. Thus $\tilde{M}_\ve$ is a square integrable martingale. Thus, it follows from \eqref{4.63} that
\[
\tilde{M}_\ve(t) \to \tilde{M}(t) = (\tilde{w}(t),\f)_2 - (w_0,\f) + \frac{1}{2} \sum_k \lambda_k^2 \int_0^t (\Psi_k \d_x (\Psi_k \tw(s)), \d_x \f)_2 \, ds
\]
almost surely as $\ve \to 0.$ It remains to show that
\begin{equation}\label{4.65}
\tilde{M}(t) = - \sum_{k \in \Z} \lambda_k \int_0^t(\Psi_k \tw(s), \d \f)_2 d \tilde{\beta}^k(s) + \sum_k \gamma_k \int_0^t(\Psi_k f(\tw(s)),\f)_2 d \tilde{\beta}_1^k(s).
\end{equation}
To this end, we will use \cite{35}, Proposition A1. $\tilde{M}_\ve$ is a square integrable martingale, which is expressed via the stochastic integral \eqref{4.63}. For $0 \leq t_1 \leq t < \delta$ and any bounded continuous $\tilde{F}_{\ve,t}$-measurable function 
\[
\Phi: BC^0([0,t_1] \times \T) \times BC^0([0,t_1]; H^2(\T)) \times [0,1] \to \R
\]
we have 
\begin{equation}\label{4.64}
    \E [(\tilde{M}_\ve(t) - \tilde{M}_\ve(t_1)) \tilde{\Phi}_\ve] = 0
\end{equation}
\begin{multline}\label{4.65}
    \E [((\tilde{M}_\ve(t))^2 - (\tilde{M}_\ve(t_1))^2 - \sum_k \lambda_k^2 \int_{t_1}^{t} (\Psi_k \tw^\ve(s),\d_x \f)_2^2 \, ds \\
    - \sum_k \gamma_k^2 \int_{t_1}^{t} (\Psi_k f(\tw^\ve(s)), \f)_2^2 \, ds ) \tilde{\Phi}_\ve] = 0
\end{multline}
\[
\E \left[\left(\tilde{\beta}_{\ve}^k(t) \tilde{M}_\ve (t) - \tilde{\beta}_{\ve}^k(t_1) \tilde{M}_\ve(t_1) + \lambda_k \int_{t_1}^{t} (\Psi_k \tw^\ve(s), \d_x \f)_2 \, ds \right)  \tilde{\Phi}_\ve \right] = 0
\]
\[
\E \left[\left(\tilde{\beta}_{1,\ve}^k(t) \tilde{M}_\ve (t) - \tilde{\beta}_{1,\ve}^k(t_1) \tilde{M}_\ve(t_1) - \gamma_k \int_{t_1}^{t} (\Psi_k w^\ve(s), \d_x \f)_2 \, ds \right)  \tilde{\Phi}_\ve \right] = 0
\]
where 
\[
\tilde{\Phi}_\ve = \Phi\left(\tw^\ve|_{[0,t_1]}, \tilde{W}^\ve|_{[0,t_1]}\right).
\]
We may now pass to the limit as $\ve \to 0$ in \eqref{4.65} the same way as in Lemma 5.7 \cite{Gess}. Similar identities hold for the limiting processes $\tilde{w}(t)$, $\tilde{M}(t)$, $\tilde{\beta}^k(t)$ and $\tilde{\beta}^k_1$. Thus \eqref{4.65} follows from \cite{35}, Proposition A1. Thus we have the existence of the solution of the initial value problem \eqref{eq:stoch}. Non-negativity of this solution follows from non-negativity $\tilde{w}^\ve(t)$. The bounds \eqref{eq:apr:est1} follow from \eqref{4.11} and \eqref{4.13}, which are uniform in $\ve$. Similarly we have the estimate for its mass from \eqref{4.53}. This completes the proof of Theorem \ref{Thm4.1}.
\end{proof}
\section{Proof of Theorem \ref{Thm2.1}.}\label{Sec5}
\subsection{Estimates for the approximate solutions.} Let $\{j\delta, j = 1,...,N+1\}$ with $\delta = \frac{T}{N+1}$ be a partition of $[0,T]$. According to Theorem \ref{Thm3.1}, on every interval $[(j-1) \delta, j \delta)$ 
there exists a non-negative solution $v_N$ of the equation \eqref{DynamicsD} (in the sense of Definition \ref{Def3.1}) for any nonnegative $H^1(\T)$ initial condition. Similarly, it follows  from Theorem \ref{Thm4.1} that on every interval $[(j-1)\delta, j \delta]$ there is a nonnegative martingale solution $(\tilde{w}_j, \tilde{W}_j)$ of \eqref{DynamicsS} for any  nonnegative initial process from 
$L^p([0,1], \mathcal{F}_0, \mathcal{P}, H^1(\T))$, in the sense of Definition \ref{Def4.1}. Here $\mathcal{P} = \lambda_{[0,1]}$, 
\[
\tilde{W}_j = \left(\sum_{k \in \Z} \tilde{\beta}_j^k \Psi_k, \sum_{k \in \Z} \tilde{\beta}_{1,j}^k \Psi_k\right)
\]
and the processes $\tilde{\beta}_j^k$ and $\tilde{\beta}_{1,j}^k$ are defined analogously to $\tilde{\beta}^k$ and $\tilde{\beta}_{1}^k$ via \eqref{4.59} on the interval $[(j-1)\delta, j\delta]$. \\

Let us now describe the construction of the solution $w_N$ on $[0,T]$. In view of Theorem \ref{Thm4.1} there is a martingale solution $w_N^1$ on $[0,\delta)$, with the initial condition $w_N^0(0):=v_N(\delta-0)$, defined on $(\Omega^1 =[0,1], \mathcal{F}^1, \mathcal{F}^1_t, \lambda^1_{[0,1]})$ with the Wiener process $\tilde{\overline{W}}_1$ and filtration $\tilde{F}_t^1 = \sigma(\tilde{\overline{W}}_1(s), s \leq t)$ for $t \in [0,\delta)$. Introduce the new probability space $\Omega^2 := [0,1]\times 0,1] = \{(\omega_1, \omega_2), \omega_1 \in [0,1], \omega_2 \in [0,1]\},$ with the measure $\lambda^2 = \lambda_{[0,1]}(\omega_1)\lambda_{[0,1]}(\omega_2)$, $\mathcal{F}^2 = \mathcal{F}^1(\omega_1) \times \mathcal{F}^2(\omega_2)$. We define the new Wiener process $\tilde{\overline{W}}'(t, \omega_2)$ and then construct the Wiener process $\tilde{\overline{W}}(t, \omega_1, \omega_2)$ on $[0,2 \delta)$ as follows:
\[
\tilde{\overline{W}}(t, \omega_1, \omega_2) = 
\begin{cases}
    \tilde{\overline{W}}_1(t, \omega_1), \ t \in [0,\delta),\\
     \tilde{\overline{W}}'(t-\delta,\omega_2) + \tilde{\overline{W}}_1(\delta, \omega_1), \ t \in [\delta, 2\delta).
\end{cases}
\]
For any fixed $\omega_1$, the equation \eqref{DynamicsS} has a martingale solution $\eta(t)$ on $[\delta, 2 \delta)$ with the initial condition $\eta(\delta, \omega_1, \omega_2) = v_N(2\delta - 0,\omega_1)$. This way the equation \eqref{DynamicsS} has the following non-negative martingale solution on $[0,2\delta)$:
\[
W_N^2(t, \omega_1, \omega_2) = 
\begin{cases}
    W_N^1(t,\omega_1), \ t \in [0,\delta),\\
    \eta(t, \omega_1, \omega_2), \ t \in [\delta, 2\delta).
\end{cases}
\]
with filtration $\mathcal{F}_t^2 = \sigma((\tilde{\overline{W}}(s), W_N^2(s)), s \leq t)$ with $t \in [0,2 \delta)$. In other words, we have the martingale solution of \eqref{DynamicsS} on $[0, 2\delta)$ with some Wiener process $\tilde{\overline{W}}_2(t, \omega_1, \omega_2)$, defined on $\Omega^2$ for $t \in [0, 2 \delta)$. Continuing this procedure, we have the existence of a martingale solution $w_N(t)$ on $[0,T)$, defined on $\Omega^N$ with some Wiener process

\begin{equation}\label{5.1}
\overline{W}_N(t) = \left(\sum_{k \in \Z} \beta_N^k \Psi_k, \sum_{k \in \Z} \beta_{1,N}^k \Psi_k\right) 
\end{equation}
Then the equation \eqref{DynamicsS} can be rewritten in the form
\begin{multline}\label{5.2}
   (w_N(t, \cdot),\f)_2 - (w_N((j-1)\delta, \f)_2  = - \frac{1}{2} \sum_{k\in \Z} \lambda_k^2 \int_{(j-1)\delta}^t (\Psi_k \d_x(\Psi_k w_N(s,\cdot)), \d_x \f)_2 \, ds \\
   - \sum_{k\in \Z} \lambda_k \int_{(j-1)\delta}^t (\Psi_k w_N(s,\cdot), \d_x \f)_2 \, d \beta_N^k(s) + \sum_{k\in \Z} \gamma_k \int_{(j-1)\delta}^t (\Psi_k f(w_N(s,\cdot)), \f)_2 \, d \beta_{1,N}^k(s)
\end{multline}
Using Theorems \eqref{Thm3.1} and \eqref{Thm4.1} the limits $w_N((j-1)\delta, \cdot) = \lim_{t \to j\delta} v_N(t, \cdot)$ and $v_N(j\delta, \cdot) = \lim_{t \to j\delta} w_N(t,\cdot)$  exist almost surely. Furthermore, 
\[
\E \|\d_x w_N((j-1)\delta, \cdot)\|_2^p< \infty
\]
and
\[
\|\d_x v_N(j \delta, \cdot)\|_2^p < \infty
\]
for $j \in \{1, ..., N+1\}$.
In a similar way to \cite{Gess}, we define the concatenated approximate solution $u_N:[0,T) \times \T \times [0,1] \to 
[0,\infty)$ by
\begin{equation}\label{5.3}
u_N(t,\cdot):=
\begin{cases}
v_N(2t - (\gamma-1)\delta, \cdot), \ t \in [(j-1)\delta, (j- \frac{1}{2})\delta),\\
w_N(2t - j\delta, \cdot), \ t \in [(j- \frac{1}{2})\delta, j \delta), \ \ j=1, ..., N+1.
\end{cases}
\end{equation}
It follows from Theorem \ref{Thm3.1} that the mass of $v_N$ is non-increasing, that is 
\begin{equation}\label{5.4}
\int_0^L v_N(t,x) \, dx  \leq \int_0^L v_N((j-1)\delta,x) \, dx, \text{ for } t \in [(j-1)\delta, j \delta)
\end{equation}
Similarly, by Theorem \ref{Thm4.1} we get
\begin{equation}\label{5.5}
\E \left(\int_0^L w_N(t,x) \, dx \right)^p \leq e^{Ct} \E \left(\int_0^L w_N((j-1)\delta,x) \, dx \right)^p, \ t \in [(j-1)\delta, j \delta).
\end{equation}
It follows from \eqref{DynamicsD}, \eqref{DynamicsS}, \eqref{5.4} and \eqref{5.5}, for $t \in [(j-1)\delta, j \delta)$
\[
\E \left|\int_0^L w_N(t,x)\, dx\right|^p \leq e^{C\delta j} \left|\int_0^L u_0(x) \, dx \right|^p,
\]
and
\[
\E \left|\int_0^L v_N(t,x)\, dx\right|^p \leq e^{C\delta j} \left|\int_0^L u_0(x) \, dx \right|^p.
\]
So for all $t \in [0,T)$ we obtain
\begin{equation}\label{5.6}
    \E \left|w_N(t,x) \, dx \right|^p + \E \left|v_N(t,x) \, dx \right|^p + \E \left|u_N(t,x) \, dx \right|^p \leq A \left|u_0(x) \, dx \right|^p
\end{equation}
for some positive constant $A$.
\begin{proposition}
For any $p \in [2, \infty)$, there exists a constant $C>0$ such that for all $N \in \N$ we have
\[
v_N, w_N \text{ and } u_N \in L^p([0,1], \mathcal{F}, \lambda_{[0,1]}, L^\infty([0,T]; H^1(\T)))
\]
with 
\begin{multline}\label{5.7}
{\rm esssup}_{t \in [0,T)}\|u_N(t)\|_{1,2}^p + {\rm esssup}_{t \in [0,T)}\|u_N(t)\|_{1,2}^p + {\rm esssup}_{t \in [0,T)}\|u_N(t)\|_{1,2}^p  \\
+ \E \int_0^T \|v_N(t)\|_{1,2}^{p-2} \int_{v_N(t)>0} (v_N \d_x^3 v_N)^2\, dx \, dt \leq C\|u_0\|_{1,2}^p.
\end{multline}
\end{proposition}
\begin{proof}
Similarly to \cite{Gess}, using \eqref{pth_est} and \eqref{eq:apr:est1} we get
\begin{equation}\label{5.8}
    \E \|\d_x v_N(j\delta, \cdot)\|_2^p \leq e^{C_2 j \delta} \left(\|\d_x u_0\|_2^p + C_3 j \delta \left(\int_0^L u_0(x)\, dx \right)^p \right)
\end{equation}
\begin{equation}\label{5.9}
    \E \|\d_x w_N(j\delta, \cdot)\|_2^p \leq e^{C_2 j \delta} \left(\|\d_x u_0\|_2^p + C_3 j \delta \left(\int_0^L u_0(x)\, dx \right)^p \right)
\end{equation}
for $j = 0,...,N.$ Using \eqref{5.6}, \eqref{5.8}, \eqref{5.9} and Poincare inequality, we have
\begin{equation}\label{5.10}
\E \|v_N(j\delta, \cdot)\|_{1,2}^p + \E \|w_N(j\delta, \cdot)\|_{1,2}^p \leq C\|u_0\|_{1,2}^p, \ j = 0, ..., N. 
\end{equation}
for some $C>0$. Denote 
\[
Y(t):=\|v_N(t)\|_{1,2}^{p-2} \int_{v_N(t)>0} (v_N \d_x^3 v_N)^2\, dx.
\]
It follows from \eqref{pth_est}
\[
\int_0^\delta Y(t) \, dt \leq \|\d_x u_0\|_2^p \|\d_x v_N(\delta-0)\|_2^p
\]
\begin{eqnarray*}
    \nonumber & & \E \int_{\delta}^{2\delta} Y(t) \, dt \leq \E \|\d_x v_N(\delta + 0)\|_2^p -  \E  \|\d_x v_N(2 \delta - 0)\|_2^p \\
    \nonumber & & = \E \|\d_x w_N(\delta - 0)\|_2^p - \E  \|\d_x v_N(2 \delta - 0)\|_2^p \\
   & & \leq e^{C_2 \delta} \E  \|\d_x v_N(\delta - 0)\|_2^p + e^{C_2 \delta} C_3 \delta \E \left|\int_0^L v_N(\delta - 0) \, dx\right|^p - \E \|\d_x v_N(2 \delta - 0)\|_2^p
\end{eqnarray*}
This way 
\begin{multline}\label{5.11}
    \E \int_0^{2\delta} Y(t)\, dt \leq \|\d_x u_0\|_2^p + (e^{C_2\delta} - 1) \E \|\d_x v_N(\delta-0)\|_{1,2}^p + C_4\delta \|v_N(\delta-0)\|_{1,2}^p \\
    - \E \|\d_x v_N(2\delta-0)\|_2^p 
\end{multline}
for all $j= 0,...,N.$ It follows from \eqref{5.11} and \eqref{5.10} we have
\begin{equation}\label{5.12}
    \E \int_0^{2\delta} Y(t)\, dt \leq \|\d_x u_0\|_2^p + C(e^{C_2\delta} - 1)  \|u_0\|_{1,2}^p + C_5\delta \|u_0\|_{1,2}^p 
    - \E \|\d_x v_N(2\delta-0)\|_2^p.
\end{equation}

We continue in a similar way:
\begin{multline*}
    \E \int_{2\delta}^{3\delta} Y(t)\, dt \leq \E \|\d_x v_N(2 \delta + 0)\|_2^p - \E \|\d_x v_N(3 \delta - 0)\|_2^p 
    = \E \|\d_x w_N(2 \delta - 0)\|_2^p \\
    - \E \|\d_x v_N(3 \delta - 0)\|_2^p \leq e^{C_2\delta}\E \|\d_x v_N(2 \delta - 0)\|_2^p - \E \|\d_x v_N(3 \delta - 0)\|_2^p.
\end{multline*}
Thus
\begin{eqnarray*}
    \nonumber & & \E \int_0^{3\delta}Y(t)\, dt \leq \|\d_x u_0\|_2^p + C(e^{C_2\delta} - 1)  \|u_0\|_{1,2}^p + C_5\delta \|u_0\|_{1,2}^p \\
    \nonumber & & +(e^{C_2\delta} - 1) \E \|\d_x v_N(2\delta-0)\|_2^p +C_5\delta \|u_0\|_{1,2}^p - \E \|\d_x v_N(3 \delta - 0)\|_2^p \\
    \nonumber & & \leq \|\d_x u_0\|_2^p + C(e^{C_2\delta} - 1)  \|u_0\|_{1,2}^p +  C_5\delta \|u_0\|_{1,2}^p + C(e^{C_2\delta} - 1)  \|u_0\|_{1,2}^p
\end{eqnarray*}
Continuing the process on the intervals $[(j-1)\delta, j\delta]$, and using the inequality \eqref{eq:apr:est1}, Theorem \ref{Thm4.1} and Poincare inequality, we obtain \eqref{5.7}.
\end{proof}
\begin{proposition}\label{Prop5.2}
    For any $p \in [2,\infty), \ve>0$, $\kappa \in (2 \ve, 2/p) \cap (2\ve, 1/2]$ and $q \in (2/(\kappa-2\ve), \infty)$, there exists $C>0$ such that for all $N \in \N$ we have
    \[
    u_N \in L^p\left([0,1], \mathcal{F}, \lambda_{[0,1]}, B^{\frac{\kappa}{2}-\ve,q}\left([0,T], B_q^{\frac{1}{2}-2 \kappa, q}(\T)\right)\right)
    \]
    and
    \begin{equation}\label{5.12}
    \E \|u_N\|^p_{B^{\frac{\kappa}{2}-\ve,q}\left([0,T], B_q^{\frac{1}{2}-2 \kappa, q}(\T)\right)} \leq C \|u_0\|^p_{1,2} (1+ \|u_0\|^{kp}_{1,2})
    \end{equation}
\end{proposition}
\begin{proof}In order to proceed with the proof of this proposition, we start with the corresponding results for $v_N$ and $w_N$.
\begin{lemma}
    For any $p \in [2,\infty), \ve>0$ and $q \in [p,\infty)$, there exists a constant $C>0$ such that for all $N \in \N$, $j \in \{1, ..., N+1\}$ and $\kappa \in (0, \frac{2}{p})$ we have
    \[
    v_N \in L^p\left([0,1], \mathcal{F}, \lambda_{[0,1]}, B_q^{\frac{\kappa}{2}-\ve,q}\left([(j-1)\delta, j \delta), B_q^{\frac{1}{2}-2 \kappa, q}(\T)\right)\right)
    \]
    and 
    \begin{equation}\label{5.13}
        \E \left(\sum_{j=1}^{N+1} \|v_N\|^p_{B_q^{\frac{\kappa}{2}-\ve,q}\left([(j-1)\delta,j \delta), B_q^{\frac{1}{2}-2 \kappa, q}(\T)\right)} \right) \leq \|u_0\|^{(\kappa+1)p}_{L^2}(1+\|u_0\|^{\lambda-2}_{L^2})^{\kappa p}
    \end{equation}
\end{lemma}
\begin{proof}
Using Definition \ref{Def3.1}, for $(j-1)\delta \leq t_1 \leq t_2 < j \delta$ we have
\[
(v_N(t_2, \cdot) - v_N(t_1, \cdot), \f)_2 = \int_{t_1}^{t_2} \int_{\{v_N(t,\cdot)>0\}} v_N^2(\d_x^3 v_N)(\d_x \f) dx dt + \int_{t_1}^{t_2} \int_0^L l(v_N) \f \, dx dt,
\]
where $\f \in C^\infty(\T)$. Then $v_N(t_2, \cdot) - v_N(t_1, \cdot)$ generates a linear continuous functional on $H^1(\T)$. Let us proceed with the estimate of its norm:
\[
\|v_N(t_2, \cdot) - v_N(t_1, \cdot)\|_{H^{-1}(\T)} \leq \int_{t_1}^{t_2} \left( \int_{\{v_N(t,\cdot)>0\}} v_N^4(\d_x^3 v_N)^2 dx +  \int_0^L l^2(v_N) \, dx \right)^{\frac{1}{2}} \, dt
\]
We proceed with the estimate of the last term:
\begin{multline}\label{5.14}
\int_{t_1}^{t_2}  \left(\int_0^L l^2(v_N) \, dx \right)^{\frac{1}{2}}
= \int_{t_1}^{t_2}  \left(\int_0^L v_N^{2\lambda} \, dx \right)^{\frac{1}{2}} \\
\leq C L^{\frac{1}{2}} {\rm esssup}_{t \in [(j-1) \delta, j \delta)} \sup_{x \in \T} |v_N(t,x)|^\lambda (t_2-t_1) = C_1 {\rm esssup}_{t \in [0,T)}\|v_N\|_{1,2}^\lambda (t_2 - t_1)
\end{multline}
where we used the Sobolev embedding. Then, using \eqref{5.7}, we have
\[
\|v_N\|_{L^2([0,1], \mathcal{F}, \lambda_{[0,1]}; C^{\frac{1}{2}}([(j-1)\delta, j \delta]; H^{-1}(\T)))} \leq C (\|u_0\|_{1,2}^2 + \|u_0\|_{1,2}^\lambda).
\]
The rest of the proof follows the lines of Lemma 4.3 from \cite{Gess}
\end{proof}
We now proceed with the corresponding result for $w_N$.
\begin{lemma}\label{lem5.4}
   For any $p \in [2,\infty), \ve>0$ and $q \in [p,\infty)$, there exists a constant $C>0$ such that for all $N \in \N$, $j \in \{1, ..., N+1\}$ and $\gamma \in (0,1)$ we have
    \[
    w_N \in L^p\left([0,1], \mathcal{F}, \lambda_{[0,1]}, B_q^{\frac{\gamma}{2}-\ve,q}\left([(j-1)\delta,j \delta), B_q^{\frac{1}{2}-2 \gamma, q}(\T)\right)\right)
    \]
    and 
    \begin{equation}\label{5.13}
        \E \left(\sum_{j=1}^{N+1} \|w_N\|^p_{B_q^{\frac{\gamma}{2}-\ve,q}\left([(j-1)\delta,j \delta), B_q^{\frac{1}{2}-2 \gamma, q}(\T)\right)} \right)^\frac{p}{q} \leq \|u_0\|_{1,2}^p
    \end{equation} 
\end{lemma}
The proof of Lemma \ref{lem5.4} is analogous to the proof of Lemma 4.2 \cite{Gess}, and hence the proof of Proposition \ref{Prop5.2} follows the lines of Proposition 4.2 \cite{Gess}. 
\subsection{Proof of the main result: Theorem \ref{Thm2.1}}
This subsection is devoted to the proof of Theorem \ref{Thm2.1}. While the proof is conceptually close to the proof of Theorem 1.2 \cite{Gess}, we are going to highlight the differences, induced by the presence of nonlinear terms $l(u)$ and $f(u)$.
\begin{proposition}\label{Prop5.4}
    Denote
    \[
\mathcal{X}_u:= BC^0([0,T] \times \T)
    \]
    \[
\mathcal{X}_J:= L^2([0,T] \times \T) \text{ (with weak topology) }
    \]
    \[
\mathcal{X}_W:= BC^0([0,T]; H^2(\T) \times H^2(\T))
    \]
Then there exist random variables $\tilde{u}$, $\tilde{u}_N:[0,1] \to \mathcal{X}_u$, $J_N,  J:[0,1] \to \mathcal{X}_J$ and $\toW'_N, \toW :[0,1] \to \mathcal{X}_W$ with $(\tilde{u}_N, \tilde{J}_N, \toW'_N) \sim (u_N, J_N, \overline{W}_N)$, where $J_N:= \chi_{v_N>0} v_N^2(\d_x^3 v_N)$. Furthermore, there are subsequences (still indexed with $N$), such that $\tilde{u}_N(\omega) \to \tilde{u}(\omega)$ in  $\mathcal{X}_u$, $\tilde{J}_N(\omega) \rightharpoonup \tilde{J}(\omega)$ in $\mathcal{X}_J$ and $\toW'_N(\omega) \rightharpoonup \toW(\omega)$ in $\mathcal{X}_W$ for every $\omega \in [0,1]$ as $N \to \infty$. 
\end{proposition}
\begin{proof}
Once we introduce
\[
\overline{W}_N^0(t,\cdot):= 
\begin{cases}
\overline{W}_N((j-1) \delta, \cdot) \ \text{ for } t \in [(j-1)\delta, (j - \frac{1}{2})\delta)\\
\overline{W}_N(2t - j \delta, \cdot) \ \text{ for } t \in [(j-\frac{1}{2})\delta, j \delta)
\end{cases}
\]
and 
\[
\toW_N^0(t,\cdot):= 
\begin{cases}
\toW'_N((j-1) \delta, \cdot) \ \text{ for } t \in [(j-1)\delta, (j - \frac{1}{2})\delta)\\
\toW'_N(2t - j \delta, \cdot) \ \text{ for } t \in [(j-\frac{1}{2})\delta, j \delta)
\end{cases}
\]
the rest of the proof becomes analogous to the proof of Proposition 5.2 \cite{Gess}. 
\end{proof}
In a similar way we get the analog of Proposition 5.3 \cite{Gess}. In particular, the analogs of (5.4a)-(5.4c) in \cite{Gess} in our case are
\begin{eqnarray*}
   & & \beta_N^k(t) = \lambda_k^{-1}(W_N^0(t), \Psi_k)_{2,2} \\
   & & \beta_{1,N}^k(t) = \gamma_k^{-1}(W_{1,N}^0(t), \Psi_k)_{2,2} \\
   & & \tilde{\beta}_N^k(t) = \lambda_k^{-1}(\tilde{W}_N(t), \Psi_k)_{2,2} \\
   & & \tilde{\beta}_{1,N}^k(t) = \gamma_k^{-1}(\tilde{W}_{1,N}(t), \Psi_k)_{2,2} \\
   & & \tilde{\beta}^k(t) = \lambda_k^{-1}(\tilde{W}(t), \Psi_k)_{2,2} \\
   & & \tilde{\beta_1}(t) = \gamma_k^{-1}(\tilde{W_1}(t), \Psi_k)_{2,2}
\end{eqnarray*}
\begin{corollary}\label{Cor:5.5}
    For $\tilde{u}_N$, $\tilde{v}_N$, $\tilde{w}_N$ and $\tilde{u}$ in Proposition \ref{Prop5.4} we have
    \begin{eqnarray*}
   & & \|\tilde{u}_N - u\|_{BC^0([0,T] \times \T)} \to 0;\\
   & & \|\tilde{v}_N - u\|_{L^\infty([0,T] \times \T)} \to 0;\\
   & & \|\tilde{w}_N - \tilde{u}\|_{L^\infty([0,T] \times \T)} \to 0
\end{eqnarray*}
as $N \to \infty$ $\lambda_{[0,1]}$ almost surely.
\end{corollary}
The proof of this corollary is identical to Corrollary 5.4 \cite{Gess}. Furthermore, we have the analog of Proposition 5.5 \cite{Gess}:
\begin{proposition}\label{Prop5.6}
    Let $\tilde{u}_N$ and $\tilde{u}$ be given in Proposition \ref{Prop5.4}. Then there are subsequences of $\tilde{u}_N$, $\tilde{v}_N$ and $\tilde{w}_N$, still indexed with $N$, such that for any $p \in [2,\infty)$ we have $\tilde{u}_N \rightharpoonup^* \tilde{u}$, $\tilde{v}_N \rightharpoonup^* \tilde{u}$ and $\tilde{w}_N \rightharpoonup^* \tilde{u}$ in $L^p([0,1], L^\infty([0,T], H^1(\T))$ as $N \to \infty$. Furthermore, 
    \begin{equation}\label{5.20}
\E {\rm esssup}_{t \in [0,T)} \|\tilde{u}(t)\|_{1,2}^p  \leq C \|u_0\|_{1,2}^p
    \end{equation}
    for some positive constant $C$ independent of $\tilde{u}$ and $u_0$. 
\end{proposition}
Hence $\tilde{u}$ is a bounded continuous $H_w^1(\T)$ - valued process.  
\begin{proposition}\label{Prop5.7}
    Let $\tilde{u}_N, \tilde{u}, \tilde{J}_N$ and $\tilde{J}$ be as in Proposition \ref{Prop5.4}. Then the distributional derivative $\d_x^3 \tilde{u}$ satisfies $\d_x^3 \tilde{u} \in L^2(\{\tilde{u} > r\})$ for any $r>0$. Furthermore, $\tilde{J}_N = \chi_{\tilde{v}_N>0} \tilde{v}_N^2 (\d_x^3 \tilde{v}_N)$ and $\tilde{J} = \chi_{\tilde{u}>0} \tilde{u}_N^2 (\d_x^3 \tilde{})$ $\lambda_{[0,1]}$ almost surely. 
\end{proposition}
We may proceed with the scheme \eqref{DynamicsD} - \eqref{DynamicsS} to conclude that for $t \in [0,T)$ and $\delta = \frac{T}{N+1}$ 
\begin{eqnarray*}
    & & (v_N(t), \f)_2 - (u_0, \f)_2 = (v_N(t),\f)_2 + \sum_{j=1}^{\left[\frac{t}{\delta}\right]} \left( -(v_N(j\delta, \cdot), \f)_2 + \lim_{s \to j\delta} (w_N(s, \cdot), \f)_2 \right) \\
    & & + \sum_{j=1}^{\left[\frac{t}{\delta}\right]}(\lim_{s \to j \delta}(v_N(s, \cdot), \f)_2 - (w_N((j-1) \delta, \cdot), \f)_2 - (v_N(0, \cdot), \f)_2 \\
    & & = (v_N(t), \f)_2  - \left(v_N \left(\left[\frac{t}{\delta}\right] \delta , \cdot\right),\f \right)_2 + \sum_{j=1}^{\left[\frac{t}{\delta}\right]}(\lim_{s \to j \delta}(v_N(s, \cdot), \f)_2 - (v_N(j-1)\delta, \cdot), \f)_2\\
    & & + \sum_{j=1}^{\left[\frac{t}{\delta}\right]}(\lim_{s \to j \delta}(w_N(s, \cdot), \f)_2 - (w_N(j-1)\delta, \cdot), \f)_2 = \int_0^t \int_{\{v_N(s,\cdot)>0\}} v_N^2(\d_x^3 v_N)(\d_x \f) \, dx ds \\
    & & -\frac{1}{2} \sum_{k \in \Z} \lambda_k^2 \int_0^{\left[\frac{t}{\delta}\right] \delta} (\Psi_k \d_x (\Psi_k w_N(s, \cdot)), \d_x \f)_2 \, ds + \int_0^t (l(v_N),\f)_2 \, ds \\
    & & - \sum_{k \in \Z} \lambda_k \int_0^{\left[\frac{t}{\delta}\right] \delta} (\Psi_k w_N(s, \cdot), \d_x \f)_2 \, d \beta_N^k(s) + \sum_{k \in \Z} \gamma_k \int_0^{\left[\frac{t}{\delta}\right] \delta} (\Psi_k f(w_N(s,\cdot)), \f)_2 \, d \beta_{1,N}^k(s) 
\end{eqnarray*}
for all $\f \in C^\infty(\T)$. Changing the stochastic basis to 
\[
([0,1], \mathcal{F}, (\mathcal{F})_{t \in [0,T]}, \lambda_{[0,1]})
\]
for $\tilde{u}_N, \tilde{v}_N$ and $\tilde{w}_N$ we obtain
\begin{eqnarray}\label{5.21}
    & & (\tilde{v}_N(t,\cdot), \f)_2 - (u_0, \f)_2 =  \int_0^t \int_{\{\tilde{v}_N(s,\cdot)>0\}} \tilde{v}_N^2(\d_x^3 \tilde{v}_N)(\d_x \f) \, dx ds \\
    \nonumber & &  -\frac{1}{2} \sum_{k \in \Z} \lambda_k^2 \int_0^{\left[\frac{t}{\delta}\right] \delta} (\Psi_k \d_x (\Psi_k \tilde{w}_N(s, \cdot)), \d_x \f)_2 \, ds + \int_0^t (l(\tilde{v}_N),\f)_2 \, ds \\
    \nonumber & & - \sum_{k \in \Z} \lambda_k \int_0^{\left[\frac{t}{\delta}\right] \delta} (\Psi_k \tilde{w}_N(s, \cdot), \d_x \f)_2 \, d \tilde{\beta}_N^k(s) + \sum_{k \in \Z} \gamma_k \int_0^{\left[\frac{t}{\delta}\right] \delta} (\Psi_k f(\tilde{u}_N(s,\cdot)), \f)_2 \, d \tilde{\beta}_{1,N}^k(s) 
\end{eqnarray}
\begin{lemma}\label{Lem5.8}
    Suppose $\tilde{u}_N, \tilde{v}_N, \tilde{w}_N, \tilde{u}, \tilde{v}, \tilde{w}$ are given in Propositions \ref{Prop5.4} and \ref{Prop5.6}. Then for any $\f \in C^\infty(\T)$ and $t \in [0,T)$, there are subsequences, still denoted with $N$, such that
    \begin{equation}\label{5.22}
        (\tilde{v}_N(t, \cdot), \f)_2 \to (\tilde{u}(t, \cdot), \f)_2,
    \end{equation}
    \begin{equation}\label{5.23}
    \int_0^t \int_{\{\tilde{v}_N(s,\cdot)>0\}} \tilde{v}_N^2(\d_x^3 \tilde{v}_N)(\d_x \f) \, dx ds \to \int_0^t \int_{\{\tilde{u}(s,\cdot)>0\}} \tilde{u}^2(\d_x^3 \tilde{u})(\d_x \f) \, dx ds
    \end{equation}
    \begin{equation}\label{5.24}
    \sum_{k \in \Z} \lambda_k^2 \int_0^{\left[\frac{t}{\delta}\right] \delta} (\Psi_k \d_x (\Psi_k \tilde{w}_N(s, \cdot)), \d_x \f)_2 \, ds \to  \sum_{k \in \Z} \lambda_k^2 \int_0^{\left[\frac{t}{\delta}\right] \delta} (\Psi_k \d_x (\Psi_k \tilde{u}(s, \cdot)), \d_x \f)_2 \, ds
    \end{equation}
     \begin{equation}\label{5.25}
\int_0^t (l(\tilde{v}_N),\f)\, ds \to \int_0^t (l(\tilde{u}),\f)\, ds
     \end{equation}
     \begin{multline}\label{5.26}
         - \sum_{k \in \Z} \lambda_k \int_0^{\left[\frac{t}{\delta}\right] \delta} (\Psi_k \tilde{u}_N(s, \cdot), \d_x \f)_2 \, d \tilde{\beta}_N^k(s) + \sum_{k \in \Z} \gamma_k \int_0^{\left[\frac{t}{\delta}\right] \delta} (\Psi_k f(\tilde{u}_N(s,\cdot)), \f)_2 \, d \tilde{\beta}_{1,N}^k(s) \to \\
         - \sum_{k \in \Z} \lambda_k \int_0^{t} (\Psi_k \tilde{u}(s, \cdot), \d_x \f)_2 \, d \tilde{\beta}^k(s) + \sum_{k \in \Z} \gamma_k \int_0^{t} (\Psi_k f(\tilde{u}(s,\cdot)), \f)_2 \, d \tilde{\beta}_{1}^k(s)
     \end{multline}
\end{lemma}
\begin{proof}
The convergences \eqref{5.22} - \eqref{5.24} can be established the same way as in Lemma 5.7 \cite{Gess}. Furthermore, the continuity of $l(u)$ implies $(l(\tilde{v}_N),\f)_2 \to (l(\tilde{u}),\f)_2$ as $N \to \infty$ almost surely. Thus
\begin{equation*}
   \E \int_0^t \int_0^L ( (l(\tilde{v}_N) - l(\tilde{u})))^2 \, dx ds \leq 2 \E \int_0^t \int_0^L l^2(\tilde{v}_N) \f^2  \, dx ds + 2 \E \int_0^t \int_0^L l^2(\tilde{u}) \f^2  \, dx ds := J_1 + J_2.
\end{equation*}
Using \eqref{5.7} and Sobolev embedding, we have
\[
J_1 \leq C_1 \E {\rm esssup }  \sup_{x \in \T} |l^2(\tilde{v}_N(t,x))| \leq C_2 \E {\rm esssup } \|v_N\|_{1,2}^{2 \lambda} \leq C_3 \|u_0\|_{1,2}^{2 \lambda}.
\]
In view of Proposition \ref{Prop5.6}, the same estimate holds for $J_2$ as well. Thus, by Vitali's convergences theorem, we have \eqref{5.25}. Finally, the convergence \eqref{5.26} can be shown the same way as the convergence of the last two terms in \eqref{4.60} in the proof of Theorem \ref{Thm4.1}. Hence the proof of Lemma \ref{Lem5.8} follows. 
\end{proof}
We may now use Lemma \ref{Lem5.8} to pass to the limit as $N \to \infty$ in \eqref{5.21}, which concludes the proof of Theorem \ref{Thm2.1}.
\end{proof}
\section*{Acknowledgments}  The research of Oleksandr Misiats was supported by Simons Collaboration
Grant for Mathematicians No. 854856. The work of Oleksandr Stanzhytskyi and Oleksiy Kapustyan was partially supported by the National Research Foundation of Ukraine No. F81/41743 and Ukrainian Government Scientific Research Grant No. 210BF38-01.


\begin{thebibliography}{10}

\bibitem{Amann}
H.~Amann.
\newblock Compact embeddings of vector-valued sobolev and besov spaces.
\newblock {\em Glasnik Matematicki}, 35(55), 2000.

\bibitem{DalPas}
E.~Beretta, M.~Bertsch, and R.~Dal~Passo.
\newblock Nonnegative solutions of a fourth-order nonlinear degenerate
  parabolic equation.
\newblock {\em Arch. Rational Mech. Anal.}, 129(2):175--200, 1995.

\bibitem{5}
J.~Bergh and J.~Lofstrom.
\newblock {\em Interpolation spaces. An introduction}, volume~18.
\newblock North Holland Publishing Co., 1978.

\bibitem{Bernis}
F.~Bernis.
\newblock Finite speed of propagation for thin viscous flows.
\newblock {\em C. R. Acad. Sci. Paris. I Math.}, 322(12):1169--1174,
  1996.

\bibitem{Bern}
F.~Bernis and A.~Friedman.
\newblock Higher order nonlinear degenerate parabolic equations.
\newblock {\em J. Differ. Eqn.}, 83(1), 1990.

\bibitem{Bert}
A.~Bertozzi and M.~Pugh.
\newblock The lubrication approximation for thin viscous films: regularity and
  long-time behavior of weak solutions.
\newblock {\em Comm. Pure Appl. Math.}, 49(2):85--123, 1996.

\bibitem{MisStaCla}
J.~Clark, O.~Misiats, V.~Mogylova, and O.~Stanzhytskyi.
\newblock Asymptotic behavior of stochastic functional differential evolution
  equation.
\newblock {\em Electron. J. Differential Equations}, pages Paper No. 35, 21,
  2023.

\bibitem{13}
F.~Cornalba.
\newblock A priori positivity of solutions to a non-conservative stochastic
  thin film equation.
\newblock {\em ArXiv: 1811.07826}, 2018.

\bibitem{Gess2}
K.~Dareiotis, B.~Gess, M.~Gnann, and G.~Gr\"{u}n.
\newblock Non-negative {M}artingale solutions to the stochastic thin-film
  equation with nonlinear gradient noise.
\newblock {\em Arch. Ration. Mech. Anal.}, 242(1):179--234, 2021.

\bibitem{17}
B.~Davidovich, E.~Moro, and H~Slone.
\newblock Spreading of viscous fluid drops on a solid substrate assisted by
  thermal fluctuations.
\newblock {\em Phys. Rev. Let.}, 95, 2005.

\bibitem{19}
J.~Fischer and G.~Gr\"{u}n.
\newblock Existence of positive solutions to stochastic thin-film equations.
\newblock {\em SIAM J. Math. Anal.}, 50(1):411--455, 2018.

\bibitem{20}
F.~Flandole and D.~Catarek.
\newblock Martingale and stationary solutions for stohasti navier-stokes
  equation.
\newblock {\em Probability Theory and Related Fields}, 102(3):367--391, 1995.

\bibitem{Gess}
B.~Gess and M.~Gnann.
\newblock The stochastic thin-film equation: existence of nonnegative
  martingale solutions.
\newblock {\em Stochastic Process. Appl.}, 130(12):7260--7302, 2020.

\bibitem{33}
G.~Gr\"{u}n, K.~Mecke, and M.~Rauscher.
\newblock Thin-film flow influenced by thermal noise.
\newblock {\em J. Stat. Phys.}, 122(6):1261--1291, 2006.

\bibitem{MisStaHie}
M.~Hieber, O.~Misiats, and O.~Stanzhytskyi.
\newblock On the bidomain equations driven by stochastic forces.
\newblock {\em Discrete Contin. Dyn. Syst.}, 40(11):6159--6177, 2020.

\bibitem{35}
M.~Hofmanova.
\newblock Degenerate parabolic stochastic partial differential equations.
\newblock {\em Stochastic Processes and Applications}, 23(12):4294--4336, 2013.

\bibitem{Jakub}
A.~Jakubovski.
\newblock The almost sure skorokhod representation for subsequences in
  nonmetric spaces.
\newblock {\em Teor. Veroyatnost. i Primenen.}, 42(1):209--216, 1997.

\bibitem{KapTar}
O.~Kapustyan, P.~Kasyanov, and Taranets.R.
\newblock Strong solutions and trajectory attractors to the thin-film equation
  with absorption.
\newblock {\em J. Math. Anal. Appl.}, 493, 2021.

\bibitem{MisStaKap}
O.~Kapustyan, O.~Misiats, and O.~Stanzhytskyi.
\newblock Strong solutions and asymptotic behavior of bidomain equations with
  random noise.
\newblock {\em Stoch. Dyn.}, 22(6):Paper No. 2250027, 28, 2022.

\bibitem{LiuRoch}
W.~Liu and M.~Rochner.
\newblock {\em Stochastic Partial Differential Equations: An Introduction}.
\newblock Springer, 2015.

\bibitem{ManZau}
R.~Manthey and Th. Zausinger.
\newblock Stochastic evolution equations in {$L^{2\nu}_\rho$}.
\newblock {\em Stochastics Stochastics Rep.}, 66(1-2):37--85, 1999.

\bibitem{MisStaTop}
O.~Misiats, O.~Stanzhytskyi, and I.~Topaloglu.
\newblock On global existence and blowup of solutions of stochastic
  {K}eller-{S}egel type equation.
\newblock {\em NoDEA Nonlinear Differential Equations Appl.}, 29(1), 2022.

\bibitem{MisStaYip1}
O.~Misiats, O.~Stanzhytskyi, and N.~K. Yip.
\newblock Existence and uniqueness of invariant measures for stochastic
  reaction-diffusion equations in unbounded domains.
\newblock {\em J. Theoret. Probab.}, 29(3):996--1026, 2016.

\bibitem{MisStaYip3}
O.~Misiats, O.~Stanzhytskyi, and N.~K. Yip.
\newblock Asymptotic analysis and homogenization of invariant measures.
\newblock {\em Stoch. Dyn.}, 19(2):1950015, 27, 2019.

\bibitem{MisStaYip2}
O.~Misiats, O.~Stanzhytskyi, and N.~K. Yip.
\newblock Invariant measures for stochastic reaction-diffusion equations with
  weakly dissipative nonlinearities.
\newblock {\em Stochastics}, 92(8):1197--1222, 2020.

\bibitem{Sta}
M.~Rosenzweig and G.~Staffilani.
\newblock Global solutions of aggregation equations and other flows with random
  diffusion.
\newblock {\em Probability Theory and Related Fields}, 185:1219--1262, 2023.

\bibitem{Skor}
A.V. Skorokhod.
\newblock {\em Random Linear Operators}.
\newblock Reidel, 1964.

\bibitem{MisStaSta}
A.~Stanzhytsky, O.~Misiats, and O.~Stanzhytskyi.
\newblock Invariant measure for neutral stochastic functional differential
  equations with non-{L}ipschitz coefficients.
\newblock {\em Evol. Equ. Control Theory}, 11(6):1929--1953, 2022.

\bibitem{27}
A.~Taranets and A.~Shishkov.
\newblock Effect of time delay of support propagation in equations of thin
  film.
\newblock {\em Ukr. Math. J.}, 55(7):1131--1152, 2003.

\bibitem{56}
H.~Triebel.
\newblock {\em Interpolation theory, function spaces, differential operators}.
\newblock Springer-Verlag, 1976.

\end{thebibliography}
\end{document}